\documentclass[10pt,reqno]{amsart}
\usepackage{lipsum}
\usepackage{amsfonts}
\usepackage{graphicx}
\usepackage{epstopdf}
\ifpdf
  \DeclareGraphicsExtensions{.eps,.pdf,.png,.jpg}
\else
  \DeclareGraphicsExtensions{.eps}
\fi

\usepackage{amsopn}
 \usepackage{color,cases}

\usepackage{amssymb,amsmath,geometry,color}
\usepackage{epstopdf}
\usepackage{epsfig,multirow}

\newtheorem{theo}{Theorem}[section]
\newtheorem{lem}[theo]{Lemma}

\newtheorem{rem}[theo]{Remark}
\newtheorem{defi}[theo]{Definition}

\numberwithin{equation}{section}

\newcommand{\norm}[1]{\left\vert#1\right\vert}

\def\pdt2{\partial_t^2}
\def\pdx2{\partial_x^2}

\def\RR{{\mathbb{R}}}
\def\CC{{\mathbb{C}}}

\newcommand{\bT}{{\mathbb T}}

\newcommand{\fe}{\mathrm{e}}

\def\eps{\varepsilon}

\newcommand{\norms}[1]{\left\Vert#1\right\Vert_{ \mathcal{H}^{m_0}}}
 \newcommand{\normss}[1]{\left\Vert#1\right\Vert_{ L_{t}^{\infty}(\mathcal{H}^{m_0})}}
 
\begin{document}

\title[Two-scale exponential integrators for CPD]{Two-scale exponential integrators with uniform accuracy  for three-dimensional charged-particle dynamics under strong   magnetic field}

\author[B. Wang]{Bin Wang}\address{\hspace*{-12pt}B.~Wang: School of Mathematics and Statistics, Xi'an Jiaotong University, 710049 Xi'an, China}
\email{wangbinmaths@xjtu.edu.cn}

\author[Z. Miao]{Zhen Miao}
\address{\hspace*{-12pt}Z. Miao: School of Mathematics and Statistics,
          Northwestern Polytechnical University,
         710072 Xi'an,  China.}
\email{mz91127@126.com}

\author[Y. L. Jiang]{Yaolin Jiang}
\address{\hspace*{-12pt}Y. L.~Jiang: School of Mathematics and Statistics, Xi'an Jiaotong University, 710049 Xi'an, China}
\email{yljiang@mail.xjtu.edu.cn}
\date{}

\dedicatory{}

\begin{abstract}
The numerical simulation of three-dimensional charged-particle dynamics (CPD) under strong magnetic field is { a basic and challenging algorithmic task in plasma physics.} In this paper, we introduce a new methodology to design two-scale exponential integrators for three-dimensional CPD whose magnetic field's strength is inversely proportional to a dimensionless and small parameter $0<\varepsilon \ll 1$. By dealing with the transformed form of three-dimensional CPD, we linearize the magnetic field and put the residual component in a new nonlinear function which is shown to be uniformly bounded. Based on this foundation and the proposed two-scale exponential integrators, a class of novel integrators is formulated and studied. { The corresponding uniform accuracy of the proposed $r$-th order integrator    is shown to be $\mathcal{O}(h^r)$, where $r=1,2,3,4$ and the constant symbolized by $\mathcal{O}$, the time stepsize $h$ and the computation cost are all independent of $\varepsilon$. Moreover, in the case of  maximal ordering strong magnetic field, improved error bound $\mathcal{O}(\varepsilon^r h^r)$ is obtained for the proposed $r$-th order integrator.} A rigorous proof of these   uniform and improved error bounds is presented, and a numerical test is performed to illustrate the { error and efficiency} behaviour of the proposed integrators.
\\
{\bf Keywords:} Charged particle dynamics, Error estimates, Strong   magnetic field, Two-scale exponential integrators, Three-dimensional system. \\
{\bf AMS Subject Classification:} 65L05, 65L70, 78A35, 78M25.
\end{abstract}

\maketitle

\section{Introduction}
Highly accurate simulation of charged-particle dynamics (CPD) plays an important role in a broad range of applications, such as plasma physics, astrophysics, nuclear physics, inertial confinement fusion and material
processing  \cite{Arnold97,Birdsall}. For example, for the probability distribution function of a large number of  electrically charged particles, its evolution is described by the Vlasov equation which is a fundamental kinetic model of collisionless plasmas \cite{CA16}.  By the Particle-In-Cell (PIC) approach (\cite{CPC,VP3,VP4,VP5,A1,VP7,SonnendruckerBook}), the  Vlasov equation can be transformed into a system on the characteristics and the numerical simulation of CPD becomes a centre stage for this equation.

Over the past few decades, there has been a longstanding interest  in designing
numerical algorithms  for CPD and see \cite{Benettin94,Boris1970,L. Brugnano2019,VP-filbet,VP8,Hairer2017-1,Hairer2018,Hairer2023} as well as their references for example.
In this paper we formulate and analyse a class of two-scale exponential integrators for  the three-dimensional equations describing the motion
of a charged particle  under a strong inhomogeneous magnetic field (\cite{Hairer2018}):
\begin{equation}\label{charged-particle 3d}
\left\{\begin{aligned}
   & \frac{\textmd{d}}{\textmd{d} t }x(t)=v(t), \qquad \qquad \qquad \qquad \quad \ \ \  x(0)=x_0\in\RR^3,\\
   & \frac{\textmd{d}}{\textmd{d} t }v (t)=v(t)\times \frac{B(x(t))}{\eps}+E(x(t)),\ \ \
      v(0)=v_0\in\RR^3, \ \  0\leq t\leq T,
\end{aligned}\right.
\end{equation}
where $x(t):[0,T  ]\to \RR^3$ and $v(t):[0,T  ]\to \RR^3$ are respectively the unknown position and velocity of a charged particle with three dimensions,   $E(x)\in \RR^3$ is a given nonuniform electric-field function,  $B(x)\in \RR^3$ is a given magnetic field and  $0<\eps \ll 1$ is a dimensionless parameter determining the strength of the  magnetic field. In this paper, we are interested in the situation of strong magnetic field ($0<\eps \ll 1$), and we assume $\norm{B(x)}=\mathcal{O}(1)$  and the initial values $\norm{x_0}=\mathcal{O}(1), \norm{v_0}=\mathcal{O}(1)$ in the Euclidean norm $\norm{\cdot}$, where the constants symbolized by $\mathcal{O}$ here and later in the whole paper are  independent of $\varepsilon$. { Although the dynamics of charged particles  \eqref{charged-particle 3d} constitutes a relatively simple system, composed of just six ordinary differential equations (ODEs), it is widely recognized that this system plays a crucial role with broad applications in plasma physics, astrophysics, and magnetic fusion research. The numerical integration of such ODEs is  a basic and challenging algorithmic task  for particle methods in plasma physics.
For example, with   the  PIC  approach to the  Vlasov equation,  a large number of  equations are obtained and each equation is a system of  a charged particle with the scheme \eqref{charged-particle 3d}.
Therefore, while this paper focuses exclusively on the straightforward system \eqref{charged-particle 3d}, it is important to note that the methodologies proposed here are readily applicable to the Vlasov equation in conjunction with the PIC technique.}

Concerning the numerical integrators for the CPD \eqref{charged-particle 3d},  the strong magnetic field $ B(x)/\eps$ and the nonlinearities  of $B(x)$ and $E(x)$ pose formidable challenges in the numerical simulation. For those earlier methods \cite{VP2,Hairer2017-2,He2017,Ostermann15,Tao2016,Webb2014,Zhang2016} that do not focus on the case of strong magnetic field, their accuracy usually will lose when $\eps$ becomes small. Thus one  needs to adopt very small time stepsize in order to obtain accurate simulations, which is time consuming and practically prohibitive. To overcome this numerical obstacle and accelerate the calculation when the magnetic field is  strong,
various methods have been proposed and analysed in recent years.  A class of asymptotic preserving  schemes \cite{VP4,VP5,A1,A2,Chacon} has been derived for the Vlasov equation and the numerical methods   consistent with the limit equation when $\eps\rightarrow 0$ and $\eps\rightarrow \mathcal{O}(1)$. Nevertheless, as stated in \cite{VP1},   the
error behaviour of asymptotic-preserving methods is   deteriorated for certain
values of $\eps$ and  as a consequence  { the methods with uniform   accuracy  (with respect to   the parameter $\eps$)  are highly desirable.}

 Combined with  the PIC approach,  many efforts on uniformly accurate (UA) methods have been done for the Vlasov equation in a strong  magnetic field.
In \cite{VP1,CPC,VP3},  methods of uniform first and second orders were derived  for the  two-dimensional system. For three-dimensional case whose external magnetic field has constant norm, a UA scheme of   order two was developed in \cite{Zhao} and was extended to   the strong magnetic field with varying direction by using some properties of Vlasov equation. { On the other hand, concerning the UA methods   solving CPD directly,}
 two filtered Boris algorithms were formulated  in \cite{lubich19} and they were shown to have second-order uniform accuracy for  CPD under  the maximal ordering scaling  \cite{scaling1,scaling2}, i.e., $B/\eps=B(\eps x)/\eps$.
  Large-stepsize integrators with second-order uniform error bounds were studied in \cite{Hairer2022} for a more special magnetic field $B/\eps=B_0/\eps+B_1(x)$.
 Three splitting methods   have been proved  in \cite{WZ} to keep first-order uniform error bound for CPD in a general strong  magnetic field but there is a restriction on the stepsize required in the methods. In a recent paper \cite{WJ23}, for  two-dimensional CPD, a class of algorithms with improved accuracy
was developed and for  three-dimensional system under  the maximal ordering scaling, the  proposed methods  still show uniform (not improved) error bounds.
{\color{cyan}Based on these facts,  the following summaries can be drawn.
\begin{itemize}
  \item Most the existing UA algorithms of CPD  are primarily focused on the two-dimensional system or three-dimensional case
under a special strong  magnetic field (maximal ordering scaling or constant norm).
  \item The existing UA methods for three-dimensional  CPD with maximal ordering   magnetic field achieve up to fourth-order uniform accuracy.
  \item For   three-dimensional CPD with a general  strong  magnetic field, only first-order  methods with uniform error bounds have been investigated  under a restriction on the stepsize. It seems that no UA method has been developed for three-dimensional CPD which do not require special assumptions on the magnetic field.
\end{itemize}

Although there are fruitful  particle methods for practical use, a pity is that UA algorithms  for solving  three-dimensional CPD  under a general strong inhomogeneous magnetic field are deficient.  In contrast to two-dimensional CPD or three-dimensional systems   subjected to   specific magnetic fields, the three-dimensional CPD under a general strong magnetic field presents greater challenges and difficulties in the formulation of numerical algorithms.}
\begin{itemize}
  \item The three-dimensional CPD does not exhibit all the characteristics of its two-dimensional case, which brings new challenges in the development of numerical methods (see \cite{VP1,WJ23}).
  \item The strong  function  $B(x)/\eps$ will make the convergence of traditional methods depend on $1/\eps$ and its nonlinearity will prevent some processing technique of numerical  methods. For example, the well-known trigonometric/exponential integrators (\cite{ADD5,ADD4,ADD2,Hochbruck1999,Ostermann06,Ostermann,ADD3,ADD6,ADD1,WZ23})   cannot be   applicable to the three-dimensional CPD \eqref{charged-particle 3d}.
  \item The general form of $B(x)$ does not have special properties possessed by maximal ordering scaling $B(\eps x)$ and this rejects the feasibility of those algorithms and analysis presented in \cite{Hairer2022,lubich19,WJ23}.
\end{itemize}

To overcome these challenges and make the uniform error estimates  go smoothly,
this paper proposes a new methodology  to the formulation of  integrators for the three-dimensional CPD \eqref{charged-particle 3d} under a general strong inhomogeneous magnetic field.
{ Compared with the existing schemes, the main novelties and contributions are as follows.
\begin{itemize}
  \item The proposed integrators successfully achieve uniform accuracy when solving   three-dimensional CPD under a general strong magnetic field. Four uniformly accurate   integrators of orders up to  four are obtained in this paper, which rewards the integrators  with very  high precision.

  \item One more interesting finding is that for  CPD  with a maximal ordering  strong magnetic field ($B/\eps=B(\eps x)/\eps$), improved uniform error bounds are shown for the proposed methods.
  More specifically, an $r$-th order integrator  offers the improved error bound $\mathcal{O}(\varepsilon^r h^r)$ with the stepsize $h$ and $r=1,2,3,4.$

  \item To achieve these objectives, we deal with the CPD \eqref{charged-particle 3d} carefully by linearizing  the magnetic field, removing the residual component into a new nonlinear  function and  employing two-scale exponential integrators. We shall present and analyse the detailed properties of transformed systems. One vital property of the new nonlinearity can be deduced. Drawing upon this insight  and the stiff order conditions of exponential integrators, uniform and improved error bounds are derived.
\end{itemize}}

The rest of this paper is organized  as follows. In Section \ref{sec:2}, we present the formulation of  novel integrators.
 { Uniform error bounds of the proposed integrators are provided}   and rigorously proved in  Section \ref{sec:3}.
In section \ref{sec:4}, a numerical test is carried out to show the error behaviour of the proposed methods { and the efficiency in comparison with some well-known methods.}
The last section includes the   conclusions  of this paper.

 \section{Formulation of the integrator}\label{sec:2}

We have noticed that the part $ B(x)/\eps$  appeared  in the equation \eqref{charged-particle 3d} brings two difficulties in the formulation of integrators. The first is that $B(x)$ depends on the solution $x$  and classical   trigonometric/exponential integrators are not applicable. The second is that $1/\eps$ has   a very negative influence on the accuracy of numerical methods.
To obtain novel methods with uniform error bounds, we need to  deal with the CPD \eqref{charged-particle 3d} carefully.
In the formulation of numerical schemes, we
first make some reconstructions of the system to get a new equation with some nice properties. Then
we consider  numerical integration for the new system and  uniform error bounds can be deduced by using those derived properties and the analysis of exponential integrators.

\subsection{Reconstructions of the system}\label{sub21}
In this subsection, we  make some   reconstructions of the  CPD \eqref{charged-particle 3d} in order to get an equivalent system with some important properties.

{
Considering a scaling to the variable $v(t)$:
 $ w (t ):=\eps  v(t)$,
then the CPD \eqref{charged-particle 3d}  is equivalent to the following  system
\begin{equation}\label{H model problem2}
\left\{\begin{aligned}
   & \frac{\textmd{d}}{\textmd{d} t }  x (t )=\frac{w (t ) }{\eps }, \quad\qquad \qquad \qquad \qquad \ \ \ \ \    x (0)=x_0,\\
   & \frac{\textmd{d}}{\textmd{d} t }w (t )=w (t )\times \frac{B( x (t ))}{\eps }+\eps E( x (t )),\ \ \
      w (0)=\eps  v_0,\ \ 0\leq t \leq T.
\end{aligned}\right.
\end{equation}}Compared with the original CPD \eqref{charged-particle 3d}, this system has a small  nonlinearity $\eps E(x (t ))$ in the  interval $[0, T]$. Such cases of other systems with small nonlinear functions have also been well studied in  \cite{Bao23,Bao21} and some numerical schemes were proved to  offer improved uniform error bounds. { These existing publications motivate the first transformation \eqref{H model problem2}.}

In what follows, we extract a constant magnetic field from $B(x )$ and get a  linearized system.  To this end, introduce the notation
$B(x ):=(b_1( x ),b_2(  x ),b_3(  x ))^\intercal\in \RR^3$ and $$\widehat B (x ):= \begin{pmatrix}
                     0 & b_3(x ) & -b_2(x ) \\
                     -b_3(x ) & 0 & b_1(x ) \\
                     b_2(x ) & -b_1(x ) & 0 \\
                  \end{pmatrix}.
$$
%
%
%
%
%
%
%
%
%
{ Extracting a constant matrix $\widehat  B_0:= \widehat B(x_0)$ from $\widehat B(x)$, we formulate \eqref{H model problem2} as} {
\begin{equation}\label{charged-particle2 ccc}
  \left\{\begin{aligned}
   & \frac{\textmd{d}}{\textmd{d} t }  x (t )= w (t )/\eps,\ \ \ \quad   \qquad \qquad \qquad   \ \ \ \    x (0)=x_0,\\
   &\frac{\textmd{d}}{\textmd{d} t }w (t )=   \widehat  B_0  w (t )/\eps  +  F( x (t ),w (t )),\ \ \
    w (0)=\eps v_0,\ \ 0\leq t \leq T,
\end{aligned}\right.
\end{equation} where the new nonlinearity is defined by $F( x(t ) ,w(t ) ):=\frac{\widehat  B( x (t ))-\widehat  B_0}{\eps } w (t )  +\eps E( x (t )).$ In order to make the expression be more concise, we introduce two simple notations $$s_0(t \widehat B_0)=\exp(t \widehat B_0),\ \ \ s_1(t \widehat B_0)=(\exp(t \widehat B_0)-I_3)/\widehat B_0.$$
Defining two filtered variables:
$$q(t):=x(t)+s_1(-t \widehat  B_0/\eps)w(t),\ \ \ p(t):=s_0(-t \widehat  B_0/\eps)w(t),$$
 the   linearized system (\ref{charged-particle2 ccc}) is equivalent to a system of differential equations for $q(t)$ and $p(t)$:
\begin{equation}\label{H2 model problem}\left\{\begin{aligned}\frac{\textmd{d}}{\textmd{d} t }  q (t )=&s_1(-t \widehat B_0/\eps)F\big( q(t)+s_1(t \widehat B_0/\eps)p(t),s_0(t \widehat B_0/\eps)p(t)\big), \ \ \ q (0)=x_0,\\
\frac{\textmd{d}}{\textmd{d} t } p(t )=&s_0(-t \widehat B_0/\eps)F\big( q(t)+s_1(t \widehat B_0/\eps)p(t),s_0(t \widehat B_0/\eps)p(t)\big), \ \ \
p(0)=\eps v_0,\ \ 0\leq t \leq T.\end{aligned}\right.
\end{equation}}Looking carefully at the nonlinearity, it is seen that $t /\eps $ now only appears in the functions $s_0(\pm t  \widehat  B_0/\eps )$ and $s_1(\pm t  \widehat  B_0/\eps )$.
With the help of $\widehat B_0$, these functions can be computed by  (\cite{lubich19})
\begin{equation*}\begin{aligned}&s_0(\pm t  \widehat  B_0/\eps )=I_3\pm\frac{\sin(  bt /\eps )}{b}\widehat  B_0+\frac{1-\cos( bt /\eps )}{b^2}\widehat  B_0^2,\\
&s_1(\pm t  \widehat  B_0/\eps )=\pm\frac{\sin(  bt /\eps )}{b}I_3+\frac{1-\cos( bt /\eps )}{b^2}\widehat  B_0,\end{aligned}\end{equation*}
where $b:=\norm{B( x_0)}$. From this form, it follows that
these two functions are both periodic in  $ t /\eps $ and for simplicity, we assume that the period is of  $ 2\pi$ in this paper. For other period, it can be changed into $ 2\pi$-period by some
 transformation in time $t$.

 With the above preparations and by isolating the fast time variable  $ t /\eps $ as another variable $\tau$,  we introduce another two functions
$Q(t ,\tau), P(t ,\tau)\in\RR^3$ which are required to  satisfy the following partial differential equations (PDEs)
\begin{equation}\label{2scalenew}\left\{\begin{split}
  &\partial_{t }Q(t ,\tau)+\frac{\partial_\tau Q(t ,\tau)}{\eps  }=s_1(-\tau \widehat B_0)F\big( Q(t ,\tau)+s_1(\tau \widehat B_0)P(t ,\tau),s_0(\tau \widehat B_0)P(t ,\tau)\big),\\
    &\partial_{t }P(t ,\tau)+\frac{\partial_\tau P(t ,\tau)}{\eps  }=s_0(-\tau \widehat B_0)F\big( Q(t ,\tau)+s_1(\tau \widehat B_0)P(t ,\tau),s_0(\tau \widehat B_0)P(t ,\tau)\big),
    \end{split}\right.
\end{equation}
where  $t \in[0,T]$ and $\tau\in[-\pi,\pi]$.
 The initial data $Q(0,\tau), P(0,\tau)$ is free except that $Q(0,0)=x_0,   P(0,0)=\eps v_0.$
 Combining  \eqref{2scalenew} with \eqref{H2 model problem}, it is deduced that when $\tau=t /\eps $, these two systems are identical, i.e.,
 $$Q(t ,t /\eps )= q(t ),\quad P(t ,t /\eps )=p(t ),\ \ \ t \in[0,T].$$
 Letting
\begin{equation}\label{ftau} f_\tau(U):=\left(
                      \begin{array}{c}
                     s_1(-\tau \widehat B_0)F\big( Q+s_1(\tau \widehat B_0)P,s_0(\tau \widehat B_0)P\big) \\
                      s_0(-\tau \widehat B_0)F\big( Q+s_1(\tau \widehat B_0)P,s_0(\tau \widehat B_0)P\big)\\
                      \end{array}
                    \right)\    \mathrm{with}\    U:=(Q^\intercal,P^\intercal)^\intercal,\end{equation}
 the above system \eqref{2scalenew} can be expressed
in a compact form \begin{equation}\label{2scale compact}\partial_{t }U(t ,\tau)+\frac{1}{\eps  }\partial_\tau U(t ,\tau)=f_\tau(U(t ,\tau)),\ \  t \in[0,T],\ \ \tau\in\bT:=[-\pi,\pi],\end{equation}
{ with the requirement $U(0,0)=(x_0^\intercal,\eps v_0^\intercal)^\intercal.$}

It is remarked that the new variable $\tau$ of \eqref{2scale compact} offers a free degree for designing the initial { data} $U(0,\tau)$.
With the arguments of  \cite{autoUA}, a so called   ``fourth-order  initial data"  will be derived  in the rest part of this subsection.
Before going further, we need the following notations
 $$L:=\partial_\tau,\ \ \Pi \vartheta(\tau):=\frac{1}{2\pi}\int_{-\pi}^{\pi}\vartheta(\tau)d\tau,\  \ A:=L ^{-1} (I-\Pi),$$
 where $\vartheta(\tau)$ is a periodic function   on $\bT$. We note here that $\Pi,\ A$ are bounded on $C^0(\mathbb{T};H^{\sigma})$ for any { Sobolev} space  $H^{\sigma}$ with $\sigma\geq 0$.
 The new initial data is constructed based on  the Chapman-Enskog expansion   (\cite{autoUA}) of $U(t ,\tau)$:
   $$U(t ,\tau)=\underline{U}(t )+\kappa(t ,\tau)\ \ \textmd{with}\ \ \underline{U}(t )=\Pi U(t ,\tau),\ \ \Pi\kappa(t ,\tau)=0.$$
Keeping the scheme \eqref{2scale compact} in mind, these two composes satisfy the following differential equations
   \begin{equation*} \begin{aligned}
&\partial_{t }\underline{U}(t )=\Pi f_\tau\big(\underline{U}(t )+\kappa(t ,\tau)\big),\ \
\partial_{t } \kappa(t ,\tau)+\frac{1}{\eps }
 \partial_\tau \kappa(t ,\tau)=(I-\Pi)f_\tau\big(\underline{U}(t )+\kappa(t ,\tau)\big).
\end{aligned}
\end{equation*}
The second one can be expressed as
  \begin{equation}\label{kapa de}
\kappa(t ,\tau)=\eps  A f_\tau\big(\underline{U}(t )+\kappa(t ,\tau)\big)-\eps  L^{-1}\big(\partial_{t }\kappa(t ,\tau)\big).
\end{equation}
 We seek an expansion in powers of $\eps $ for the composer $\kappa$  (\cite{autoUA}):
 \begin{equation}\label{kapa}
\kappa(t ,\tau)=\eps  \kappa_1(\tau,\underline{U}(t ))+\eps^{2 } \kappa_2(\tau,\underline{U}(t ))+\eps^{3 }  \kappa_3(\tau,\underline{U}(t ))+\mathcal{O}(\eps^{4 } ),
\end{equation}
where $\mathcal{O}$ denotes the term uniformly bounded with the   Euclidean  norm.
 Inserting \eqref{kapa} into \eqref{kapa de}, taking the Taylor series of $f_\tau$ at $\underline{U}$ and comparing the coefficients   of $\eps^{j } $ with $j=1,2,3$, we can get ($(t )$ is omitted for conciseness)
  \begin{equation}\label{kapa re2}\begin{aligned}
\kappa_1(\tau,\underline{U})=&Af_\tau(\underline{U}),\ \ \kappa_2(\tau,\underline{U})=A\partial_Uf_\tau(\underline{U})Af_\tau(\underline{U})-A^2\partial_U f_\tau(\underline{U})\Pi f_\tau(\underline{U}),\\
\kappa_3(\tau,\underline{U})=&A\partial_Uf_\tau(\underline{U}) A\partial_Uf_\tau(\underline{U})Af_\tau(\underline{U})-A\partial_Uf_\tau(\underline{U})A^2\partial_U f_\tau(\underline{U})\Pi f_\tau(\underline{U})\\&
+\frac{1}{2}  A \partial^2_Uf_\tau(\underline{U})\big( Af_\tau(\underline{U}), Af_\tau(\underline{U})\big)- A^2\partial_U^2f_\tau(\underline{U})\big(\Pi f_\tau(\underline{U}),Af_\tau(\underline{U})\big)\\
&-
A^2\partial_Uf_\tau(\underline{U})A\partial_U f_\tau(\underline{U})\Pi f_\tau(\underline{U})+A^3\partial_U^2f_\tau(\underline{U})\big(\Pi f_\tau(\underline{U}),\Pi f_\tau(\underline{U})\big)\\
&+A^3\partial_U f_\tau(\underline{U})\Pi \partial_U f_\tau(\underline{U})\Pi f_\tau(\underline{U})-A^2\partial_U f_\tau(\underline{U})\Pi \partial_U f_\tau(\underline{U})A f_\tau(\underline{U}).
\end{aligned}
\end{equation}
Based on these results, the ``fourth-order  initial data" is defined by
 \begin{equation}\label{inv2}\begin{aligned}U_0(\tau):=&U(0 ,\tau)=\underline{U}(0)+\kappa(0 ,\tau)\\
 =&\underline{U}(0)+\eps \kappa_1(\tau,\underline{U}(0) )+\eps^{2 } \kappa_2(\tau,\underline{U}(0) )+\eps^{3 } \kappa_3(\tau,\underline{U}(0))+\mathcal{O}(\eps^{4 } ),\end{aligned}\end{equation}
where $\underline{U}(0)$ is chosen such that $U_0(0)=(x_0^\intercal,\eps v_0^\intercal)^\intercal.$
{
The proper $\underline{U}(0)$ can be obtained through an iteration:
\begin{equation}\label{boundqj}\underline{U} ^{[0]}(0):=U_0(0),\ \ \underline{U}^{[j]}(0)=U_0(0)-\sum_{l=1}^j\eps^{ j-l+1} \kappa_l\big(0,\underline{U}^{[j-l]}(0)\big),\quad j=1,2,3.\end{equation}
Dropping  the $\mathcal{O}(\eps^{4 } )$-term, the ``fourth-order  initial data" is finally given by
 \begin{equation}\label{inv}U_0(\tau):=\underline{U}^{[3]}(0)+\eps \kappa_1\big(\tau,\underline{U}^{[2]}(0) \big)+\eps^{2 } \kappa_2\big(\tau,\underline{U}^{[1]} (0)\big)+\eps^{3 } \kappa_3\big(\tau,\underline{U}^{[0]}(0)\big).\end{equation}}

  \begin{rem}
  It is noted that this initial data is used for  fourth-order time integrators. For lower-order methods, the initial data can be weakened into a simple form. For example, third-order  method only needs $U_0(\tau)=\underline{U}^{[2]}(0)+\eps \kappa_1\big(\tau,\underline{U}^{[1]}(0) \big)+\eps^{2 } \kappa_2\big(\tau,\underline{U}^{[0]}(0) \big)$ in practical computations.
  \end{rem}

\subsection{Full-discretization}
In this subsection, we shall present the full-discretization for the transformed PDEs \eqref{2scale compact} with the new designed initial value \eqref{inv}. Since the equation is   periodic in $\tau\in\bT:=[-\pi,\pi]$, Fourier pseudospectral method is a very suitable choice because of its high accuracy (\cite{Shen}).   After applying the Fourier pseudospectral method in $\tau$, a system of ordinary differential equations (ODEs) is derived and then exponential integrators are considered.

To use the Fourier pseudospectral method in the variable $\tau\in\bT$, we consider  the trigonometric
polynomials
$$U^{\mathcal{M}}(t ,\tau)=\big(U_1^{\mathcal{M}}(t ,\tau),
U_2^{\mathcal{M}}(t ,\tau),\ldots,U_6^{\mathcal{M}}(t ,\tau)\big)^\intercal\in \RR^6$$
with
\begin{equation*}
\begin{array}[c]{ll}
 U_j^{\mathcal{M}}(t ,\tau)=\sum\limits_{k\in \mathcal{M}}
 \widehat{U}_{k,j}(t )\mathrm{e}^{\mathrm{i} k \tau },\ \  (t ,\tau)\in[0,T]\times
[-\pi,\pi],
\end{array}
\end{equation*}
where  $j=1,2,\ldots,6$, $k\in \mathcal{M}: =\{-N_\tau/2,-N_\tau/2+1,\ldots,N_\tau/2-1\}$  with an even positive integer $N_\tau$ and  $ (\widehat{U}_{k,j}(t ))_{k\in \mathcal{M}}$ are referred to the discrete Fourier coefficients of $ U_j^{\mathcal{M}}$, i.e., $$\widehat{U}_{k,j}(t )=\frac{1}{N_\tau}\sum_{l=-N_\tau/2}^{N_\tau/2-1}U_j^{\mathcal{M}}(t , 2\pi l/N_\tau )\exp(-\mathrm{i}k 2\pi l/N_\tau).$$
Based on the scheme of \eqref{2scale compact}, we require the trigonometric
polynomials to satisfy
\begin{equation}\label{F-PDE}\partial_{t }U^{\mathcal{M}}(t ,\tau)+\frac{1}{\eps  }\partial_\tau U^{\mathcal{M}}(t ,\tau)=f_\tau\big(U^{\mathcal{M}}(t ,\tau)\big),\ \  \ (t ,\tau)\in[0,T]\times
[-\pi,\pi].\end{equation}
Collecting all the coefficients $\widehat{U}_{k,j}$
in a $D:=6 N_{\tau}$ dimensional
coefficient vector\footnote{In this paper,    for all the
vectors and diagonal matrices with the same dimension as $\widehat{\mathbf{U}}$, we use the same expression as \eqref{vec form} to denote their components.}
 {  \begin{equation} \label{vec form} \begin{aligned}\widehat{\mathbf{U}}=\Big(&\widehat{U}_{-\frac{N_{\tau}}{2},1},\widehat{U}_{-\frac{N_{\tau}}{2}+1,1},\ldots,
 \widehat{U}_{\frac{N_{\tau}}{2}-1,1},\widehat{U}_{-\frac{N_{\tau}}{2},2},\widehat{U}_{-\frac{N_{\tau}}{2}+1,2},\ldots,
 \widehat{U}_{\frac{N_{\tau}}{2}-1,2},\\
 &\ldots,\widehat{U}_{-\frac{N_{\tau}}{2},6},
 \widehat{U}_{-\frac{N_{\tau}}{2}+1,6},\ldots,
 \widehat{U}_{\frac{N_{\tau}}{2}-1,6}\Big)^\intercal\in \CC^D,\end{aligned}
\end{equation}}we get a system of  ODEs
\begin{equation}\label{2scale Fourier-f} \begin{aligned}
&\frac{\textmd{d}}{\textmd{d}t }\widehat{\mathbf{U}}(t )=\mathrm{i}\Omega\widehat{\mathbf{U}}(t )
+ \mathbf{F}(\widehat{\mathbf{U}}(t )),\ \ \ t \in[0,T],
 \end{aligned}
\end{equation}
where  $\Omega=\textmd{diag}(\Omega_1,\Omega_2,\ldots,\Omega_6)$ with
$\Omega_1=\Omega_2=\cdots=\Omega_6:=\frac{1}{\eps }\textmd{diag}\big(\frac{N_{\tau}}{2}-1,
  \frac{N_{\tau}}{2}-2,\ldots,-\frac{N_{\tau}}{2}\big),$ and $$ \mathbf{F}(\widehat{\mathbf{U}}(t )):=\left(
                      \begin{array}{c}
                     \mathcal{F} \mathbf{S}^{-} F\big(\mathcal{F}^{-1}\widehat{\mathbf{U}}^{\textmd{part1}}(t )+ \mathbf{S}^{+}\mathcal{F}^{-1}\widehat{\mathbf{U}}^{\textmd{part2}}(t ), \mathbf{E}^{+}\mathcal{F}^{-1}\widehat{\mathbf{U}}^{\textmd{part2}}(t )\big)\\
                     \mathcal{F} \mathbf{E}^{-}F\big(\mathcal{F}^{-1}\widehat{\mathbf{U}}^{\textmd{part1}}(t )+ \mathbf{S}^{+}\mathcal{F}^{-1}\widehat{\mathbf{U}}^{\textmd{part2}}(t ), \mathbf{E}^{+}\mathcal{F}^{-1}\widehat{\mathbf{U}}^{\textmd{part2}}(t )\big)\\
                      \end{array}
                    \right)$$ with $ \widehat{\mathbf{U}}(t ):=\Big((\underbrace{\widehat{\mathbf{U}}^{\textmd{part1}}(t )}_{\frac{D}{2}\  \textmd{dimensions}})^\intercal,(\underbrace{\widehat{\mathbf{U}}^{\textmd{part2}}(t )}_{\frac{D}{2}\  \textmd{dimensions}})^\intercal\Big)^\intercal.$
 Here the notations $\mathbf{E}^{\pm}$ and $\mathbf{S}^{\pm}$ are defined by
 \begin{equation*}
\begin{array}[c]{ll}%
\mathbf{E}^{\pm}=&I_3\otimes I_{N_{\tau}} \pm \Big(\frac{I_3}{b}\otimes  \textmd{diag}\big(\sin(  2b l\pi /N_\tau )\big)_{l\in \mathcal{M}}\Big) (\widehat  B_0 \otimes I_{N_{\tau}}) \\&
+\Big(\frac{I_3}{b^2}\otimes\textmd{diag}\big(1-\cos( 2b l\pi /N_\tau )\big)_{l\in \mathcal{M}}\Big) ( \widehat  B_0^2\otimes I_{N_{\tau}}),\\
\mathbf{S}^{\pm}=&  \pm \Big(\frac{I_3}{b}\otimes  \textmd{diag}\big(\sin( 2b l\pi /N_\tau )\big)_{l\in \mathcal{M}}\Big) (I_{3} \otimes I_{N_{\tau}}) \\&
+\Big(\frac{I_3}{b^2}\otimes\textmd{diag}\big(1-\cos( 2b l\pi /N_\tau )\big)_{l\in \mathcal{M}}\Big) ( \widehat  B_0 \otimes I_{N_{\tau}}),
\end{array}\end{equation*} and
$\mathcal{F} \widehat{\mathbf{U}}^{\textmd{part1}}$ (or $\mathcal{F}^{-1}\widehat{\mathbf{U}}^{\textmd{part1}}$) denotes the discrete Fourier
transform (or the inverse discrete Fourier
transform) acting on each $N_{\tau}$-dimensional
coefficient vector $ \widehat{\mathbf{U}}^{\textmd{part1}}_{:,j}$ of $
\widehat{\mathbf{U}}^{\textmd{part1}}$, where $$\widehat{\mathbf{U}}^{\textmd{part1}}_{:,j}=\Big(
\widehat{\mathbf{U}}^{\textmd{part1}}_{-\frac{N_{\tau}}{2},j},\widehat{\mathbf{U}}^{\textmd{part1}}
_{-\frac{N_{\tau}}{2}+1,j},\ldots,
\widehat{\mathbf{U}}^{\textmd{part1}}_{\frac{N_{\tau}}{2}-1,j}\Big)^\intercal\quad \textmd{for}\quad  j=1,2,3.$$
The initial condition of \eqref{2scale Fourier-f}  is stated as
 \begin{equation}\label{inv full}\widehat{\mathbf{U}}^0:=\Big(\big(\mathcal{F} (U_0( 2l\pi/N_\tau)_{l\in \mathcal{M}})^{\textmd{part1}}\big)^\intercal,\big(\mathcal{F} (U_0( 2l\pi/N_\tau)_{l\in \mathcal{M}})^{\textmd{part2}}\big)^\intercal\Big)^\intercal.\end{equation}

Based on the above preparations, we are now in a position to present the   full-discretization of the CPD \eqref{charged-particle 3d}.
Before that, we sum up the different systems introduced above and present them  in Table \ref{tab1}.
\begin{table}[t!]
\renewcommand{\arraystretch}{1.9}
\centering
\begin{tabular}
[c]{|c|c|c|c|}\hline    & Equation & Initial data & Time length\\\hline
\eqref{charged-particle 3d}: original CPD&$\dot{x}=v,  \dot{v}=v\times \frac{B(x)}{\eps}+E(x)$ &$x_0,v_0$ &$[0,T]$\\\hline
\eqref{charged-particle2 ccc}:  linearized CPD&$\dot{x} =\frac{w }{\eps },  \dot{w} =  \frac{\widehat  B_0}{\eps } w   +F( x ,w )$ &$x_0,\eps v_0$ &$[0,T]$\\\hline
\eqref{2scale compact}: two-scale system&$\partial_{t }U(t,\tau)+\frac{1}{\eps  }\partial_\tau U(t,\tau)=f_\tau(U(t,\tau))$ &\eqref{inv}: $U_0(\tau)$ &$[0,T]$\\\hline
\eqref{2scale Fourier-f}: semi-discrete system&$\frac{\textmd{d}}{\textmd{d}t }\widehat{\mathbf{U}}(t )=\mathrm{i}\Omega\widehat{\mathbf{U}}(t )
+ \mathbf{F}(\widehat{\mathbf{U}}(t ))$ &\eqref{inv full}: $\widehat{\mathbf{U}}^0$ &$[0,T]$\\\hline
\end{tabular}
\caption{Different systems considered in Section \ref{sec:2}. }%
\label{tab1}%
\end{table}

\begin{defi}\label{dIUA-PE-F}(\textbf{Fully discrete scheme})
Choose a time step size $h>0$ and a positive even number $N_{\tau}>0$.  The full-discretization of the CPD \eqref{charged-particle 3d} is defined as follows.
   \begin{itemize}
\item Step 1:
   After those  reconstructions of the  CPD \eqref{charged-particle 3d} and the application of Fourier pseudospectral method in $\tau$, we get a semi-discrete equation \eqref{2scale Fourier-f} with the initial data $\widehat{\mathbf{U}}^0$ \eqref{inv full}.

\item Step 2:  For solving this obtained system, $s$-stage explicit exponential integrators \cite{Ostermann} are considered,  that is for $M=\mathrm{i}\Omega$  \begin{equation}\label{ei-sc}
\begin{array}[c]{ll}%
 \widehat{\mathbf{U}}^{ni}&=\exp( c_{i} hM)   \widehat{\mathbf{U}}^{n}+h\textstyle\sum\limits_{j=1}^{i-1}\bar{a}_{ij}(hM)\mathbf{F}( \widehat{\mathbf{U}}^{nj}),\ \
i=1,2,\ldots,s,\\
 \widehat{\mathbf{U}}^{n+1}&=\exp( hM) \widehat{\mathbf{U}}^{n}+h\textstyle\sum\limits_{j=1}^{s}\bar{b}_{j}(hM)\mathbf{F}( \widehat{\mathbf{U}}^{nj}),\ \
 \quad n=0,1,\ldots,T/h-1.
\end{array}\end{equation}
Here  $0<h<1$ is the time stepsize and $s\geq1$ is the stage of the   integrator with the coefficients $c_i$, $\bar{a}_{ij}(h M)$ and
$\bar{b}_{j}(h M)$. This produces the approximation of  \eqref{2scale Fourier-f} $$\widehat{\mathbf{U}}^{n}\approx \widehat{\mathbf{U}}(nh),$$ where $n=1,2,\ldots,T/h$.

\item Step 3:  The numerical solution of the two-scale system \eqref{2scale compact} is given by
  \begin{equation}\label{nsfor2scale}\begin{array}[c]{ll}
  &\mathbf{U}_j^{ni}=\sum\limits_{k\in \mathcal{M}} \widehat{\mathbf{U}}^{ni}_{k,j}
\exp(\mathrm{i}k (n+c_i)h/\eps  )\approx U_j(nh+c_ih, (n+c_i)h/\eps),\ \ i=1,2,\ldots,s,\\
&\mathbf{U}_j^{n}=\sum\limits_{k\in \mathcal{M}} \widehat{\mathbf{U}}^{n}_{k,j}
\exp(\mathrm{i}k nh/\eps  )\approx U_j(nh, nh/\eps),\ \ j=1,2,\ldots,6,\ \ n=1,2,\ldots,T/h.\end{array}\end{equation}

 \item Step 4: Letting $$ q ^n:=(\mathbf{U}_1^{n},\mathbf{U}_2^{n},\mathbf{U}_3^{n})^\intercal,\quad p^n:=(\mathbf{U}_4^{n},\mathbf{U}_5^{n},\mathbf{U}_6^{n})^\intercal,$$
 the numerical approximation of the  linearized CPD \eqref{charged-particle2 ccc} is defined as
 $$ x ^n:=q^n+s_1(nh\widehat  B_0/\eps)p^n\approx    x (nh),\  w^n:=s_0(nh\widehat  B_0/\eps)p^n\approx w(nh)$$
for $n=1,2,\ldots, T/h.$

 \item Step 5: The full-discretization $x^{n} $ and $v^{n} $  of the original CPD  \eqref{charged-particle 3d}  is formulated as  \begin{equation*} \begin{aligned}
&  x ^n\approx x(nh ),\ \ \ v^{n}:= w^n/\eps\approx v(nh )\ \ \
\textmd{for} \ \ n=1,2,\ldots,T/h.
 \end{aligned}
\end{equation*}

   \end{itemize}
\end{defi}

{
\begin{rem}
Compared with traditional methods, the scheme of our integrators given in this paper is more complicated. For example, the two-scale method   enlarges the dimension of the original system and this usually adds some computation cost. Fortunately,
      Fourier pseudospectral method has very high accuracy and thus  in many applications the practical value of $N_\tau$ doesn't need to be large. For instance, as shown in the literature \cite{Zhao,autoUA,VP3} for the limit case of the charged-particle dynamics and the H\'{e}non-Heiles model, $N_{\tau}=16$ is enough to get the machine accuracy for discretizing the $\tau$-direction.
      Moreover, the use of Fast Fourier Transform (FFT) techniques within our integrators can enhance their efficiency. Based on these considerations, we anticipate that the performance of our integrators remains commendable, even when compared with the methods applied directly to CPD.  This efficiency will be numerically demonstrated
      by a numerical experiment  given in Section \ref{sec:4}.
\end{rem}}

\subsection{Practical integrators of orders up to  four}\label{sec:prac}
As the last part of this section, for the explicit  exponential scheme \eqref{ei-sc}, we present four practical  integrators whose accuracy ranges from first-order to fourth-order precision.  Here we use the notations $\varphi_{\rho}(z):=\int_0^1
\theta^{\rho-1}\fe^{(1-\theta)z}/(\rho-1)!d\theta$ for $\rho=1,2,3$ (\cite{Ostermann}) and $\varphi_{ij}:=\varphi_i(c_jh M)$.

\textbf{First-order integrator}. We consider the explicit exponential Euler method (\cite{Ostermann}) as the first practical integrator which can be
expressed as $$  \widehat{\mathbf{U}}^{n+1}=\exp( hM) \widehat{\mathbf{U}}^{n}+h\varphi_1(hM)\mathbf{F}( \widehat{\mathbf{U}}^{n}).$$
In this paper, we shall { denote this first-order method incorporating an exponential scheme as MO1-E}.

\textbf{Second-order integrator}.
 Then we   consider a second-order explicit exponential integrator  which is given by
\begin{equation*}
\begin{array}[c]{ll}%
 \widehat{\mathbf{U}}^{n1}&=\exp( hM/2) \widehat{\mathbf{U}}^{n}+h/2 \varphi_1(hM/2)\mathbf{F}( \widehat{\mathbf{U}}^{n}),\\
 \widehat{\mathbf{U}}^{n+1}&=\exp( hM) \widehat{\mathbf{U}}^{n}+h\varphi_1(hM)\mathbf{F}( \widehat{\mathbf{U}}^{n1}).
\end{array}\end{equation*}
This explicit  second-order method can be expressed as the form \eqref{ei-sc} with $s=2$ and we   denote it by  MO2-E.

\textbf{Third-order integrator}. For the third-order explicit integrator, we consider three-stage scheme \eqref{ei-sc}, i.e., $s=3$.
Following \cite{Ostermann},   choose the coefficients as
$c_1=0,   c_2=1/3 ,   c_3=2/3$ and
  \begin{equation*}
\begin{aligned} &\bar{a}_{21} =c_2\varphi_{1,2},\
\bar{a}_{31} =\frac{2}{3} \varphi_{1,3}- \frac{4}{9c_2}\varphi_{2,3},\ \bar{a}_{32} =\frac{4}{9c_2}\varphi_{2,3},\  \bar{b}_{1} = \varphi_1   -\frac{3}{2} \varphi_2 ,  \ \
 \bar{b}_{2} =0,\ \ \bar{b}_{3} =\frac{3}{2} \varphi_2.
\end{aligned}
\end{equation*}
This third-order method is referred as MO3-E.

\textbf{Fourth-order integrator}.
As the last example, we pay attention to the fourth-order explicit integrator. Choosing  $s=5$ and
 \begin{equation*}
\begin{array}
[c]{ll}%
&c_1=0,\ \ \qquad\qquad\quad\ c_2=c_3=c_5=\frac{1}{2}, \qquad\  c_4=1,\\
&\bar{a}_{2,1}=\frac{1}{2}\varphi_{1,2},\qquad  \ \quad \bar{a}_{3,1}=\frac{1}{2}\varphi_{1,3}-\varphi_{2,3},
\quad \ \bar{a}_{3,2}=\varphi_{2,3},\\
&\bar{a}_{4,1}= \varphi_{1,4}-2\varphi_{2,4},\ \ \bar{a}_{4,2}= \bar{a}_{4,3}=\varphi_{2,4},\qquad
\bar{a}_{5,1}= \frac{1}{2}\varphi_{1,5}-2\bar{a}_{5,2}-\bar{a}_{5,4},\\ &\bar{a}_{5,2}=  \frac{1}{2}\varphi_{2,5}-\varphi_{3,4}+\frac{1}{2}\varphi_{2,4}-\frac{1}{2}\varphi_{3,5},\qquad
\ \ \ \ \bar{a}_{5,3}=  \bar{a}_{5,2},\ \ \bar{a}_{5,4}=  \frac{1}{2}\varphi_{2,5}-\bar{a}_{5,2},\\
&\bar{b}_{1}=\varphi_{1}-3\varphi_{2}+4\varphi_{3},\ \ \bar{b}_{2}=\bar{b}_{3}=0,\qquad\ \ \ \ \ \ \bar{b}_{4}=-\varphi_{2}+4\varphi_{3},\ \ \bar{b}_{5}=4\varphi_{2}-8\varphi_{3},\\
\end{array}
\end{equation*}
an explicit  exponential method of order four is determined (\cite{Ostermann06}) and we refer to it as MO4-E.

To end this section, we   note that higher-order explicit exponential integrators can be constructed by deriving their stiff order conditions. This issue will not be discussed further  in this paper.
%

 \section{Error estimates}\label{sec:3}
 In this section, we deduce the error bound on  the fully discrete   scheme given in Definition \ref{dIUA-PE-F}.
For simplicity of notations,  we shall denote $C>0$   a generic constant independent of the time step $h$ or $\eps$ or $n$.
We   use the   norm $\norm{\cdot}$ of a finite dimensional vector or matrix
to denote the standard Euclidian norm in this paper.
\subsection{Main result}
\begin{theo} \label{UA thm2} (\textbf{Uniform error bounds})
For the CPD \eqref{charged-particle 3d} and its transformed system \eqref{2scale compact}, we make the following assumptions.
\begin{itemize}
  \item Assumption I. It is assumed that
 $B(x)$ and $E(x)$ of  \eqref{charged-particle 3d}  are  smooth   with
derivatives  bounded independently of $\eps$ on bounded domains, and we   assume boundedness independently
of $\eps$ for  the  initial position $x_0$ and velocity $v_0$.

  \item  Assumption II. For a smooth periodic function $\vartheta(\tau)$ on $\bT:=[-\pi,\pi]$,
denote the Sobolve space $\mathcal{H}^{m}(\bT)=\{\vartheta(\tau)\in\mathcal{H}^{m}:\partial^l_{\tau}\vartheta(-\pi)=\partial^l_{\tau}\vartheta(\pi),l=0,1,\ldots,m \}$.
The  initial value $U_0(\tau)$ \eqref{inv}  of the two-scale system  \eqref{2scale compact} is assumed to be uniformly bounded w.r.t. $\eps$ in the Sobolev space  $\mathcal{H}^{m_0}(\bT)$ with some $m_0\geq0.$
\end{itemize}
Under these assumptions,   for the numerical solution $x^n$ and $v^n$ from Definition \ref{dIUA-PE-F} as well as   the $r$-th order scheme MO$r$-E provided in Section \ref{sec:prac}, there exist constants $h_0, C>0$ depending on $T$, such that for $0<h<h_0$ and $n=1,2,\ldots,T/h$:
\begin{equation}\label{err res1}
\begin{aligned}
 &\norm{ x(nh ) -x^n}+\eps\norm{ v(nh ) -v^n}\leq C \big( h^r+ (2\pi/N_{\tau})^{m_0}\big),\quad r=1,2,3,4,
\end{aligned}
\end{equation}
where the constant $C$ is independent of the time step $h$ or $\eps$ or $n$.
{ If the magnetic field has maximal ordering scaling (see \cite{scaling1,lubich19,scaling2}), i.e., $B/\eps=B(\eps x)/\eps,$ the global errors are improved to be
\begin{equation}\label{err res2}
\begin{aligned}
 &\norm{ x(nh ) -x^n}+\eps\norm{ v(nh ) -v^n}\leq C \big(  \eps^rh^r+ (2\pi/N_{\tau})^{m_0}\big),\quad r=1,2,3,4.\\
\end{aligned}
\end{equation}}
\end{theo}

 Before stepping into the proof,  an important remark is presented { below}.
\begin{rem}
It is noted that the error  $(2\pi/N_{\tau})^{m_0}$ comes from the accuracy of Fourier pseudospectral method which can be improved quickly   as $N_{\tau}$ increases (\cite{Shen}). Thus this error is not the focus of our attention and this paper is mainly devoted to the error bound on the numerical methods in time direction. { Meanwhile, we remark that the above result \eqref{err res1} gives
  the uniform convergence of  an $r$-th order   scheme in $x$ since the accuracy is $\mathcal{O}(h^r)$, where the constant symbolized by $\mathcal{O}$ and the  stepsize $h$ are both independent of $\eps$.
This means that the procedure proposed in this paper provides four uniformly accurate  time integrators of up to order four for  the CPD \eqref{charged-particle 3d}. Moreover, in the case of maximal ordering magnetic field, the proposed integrators offer significantly  improved accuracy \eqref{err res2}, where the order in $\eps$ raises along with the order of the scheme.}
\end{rem}

%
%

 To prove this theorem, some properties of the reconstructed  systems are needed and we derive them in the following lemmas \ref{lem1} and \ref{lem2}. After obtaining them, the error bound can be deduced. The outline of the proof is given as follows.
  \begin{itemize}
 \item Lemma \ref{lem1}. For the  linearized CPD \eqref{charged-particle2 ccc}, we  estimate the bound of its solution and based on which, a bound is derived for the new nonlinearity
 $F(x,w)$.

 \item  Lemma \ref{lem2}.  Based on the results of  Lemma \ref{lem1},
 the solution of    two-scale system \eqref{2scale compact} and its derivatives w.r.t $t $ are estimated. It will be shown that
 the nonlinearity and the derivatives of the solution have  uniform bounds for general strong magnetic field and improved  bounds for maximal ordering scaling.

\item Lemma \ref{lem3}. To study the convergence of  fully discrete scheme, we first research the error brought from the Fourier pseudospectral method in the new variable $\tau$, which can be done by introducing a transitional integrator.


\item Lemma \ref{lem5}. We use the results of Lemmas \ref{lem1}-\ref{lem3} and stiff order conditions of exponential integrators  to derive  uniform bounds on the local errors for the transitional integrator.

\item Proof of  Theorem \ref{UA thm2}. With the above preparations and  by the standard arguments of convergence, global errors are shown.
\end{itemize}

 \subsection{Proof of the main result}
It is noted that  since the CPD \eqref{charged-particle 3d} and the  linearized CPD \eqref{charged-particle2 ccc} are both in  a finite dimensional vector space,  Euclidean norm $\norm{\cdot}$ is considered for them.
For the  linearized CPD \eqref{charged-particle2 ccc},  the following   estimate of its solution is obtained.
\begin{lem}\label{lem1}
Under the conditions of Assumption I proposed in Theorem \ref{UA thm2}, the solution of  the  linearized CPD \eqref{charged-particle2 ccc} satisfies
 \begin{equation}\label{lem1 bound1}
  \norm{ x (t )} \leq C,\ \ \  \norm{w (t )} \leq C \eps \ \ \textmd{for\ all}\  \ \ t \in[0,T].
\end{equation}
{ Moreover, we have  the following estimate of the nonlinearity
 \begin{equation}\label{lem1 bound2}\norm{F( x (t ),w(t ))}\leq C   \ \ \textmd{for\ all}\  \ \ t \in[0,T].\end{equation}
 For the   maximal ordering scaling case $B(\eps x)/\eps$, this bound is  improved to be
 \begin{equation}\label{lem1 bound3}\norm{F( x (t ),w(t ))}\leq C \eps  \ \ \textmd{for\ all}\  \ \ t \in[0,T].\end{equation}}
\end{lem}
\begin{proof}
For the  linearized CPD \eqref{charged-particle2 ccc}, with the new notation $\hat{w}(t ):=\exp(- t \widehat  B_0/\eps) w  (t )$, it reads
\begin{equation}\label{H1 model problem}\begin{aligned}\frac{\textmd{d}}{\textmd{d} t }  x (t )=&\exp( t \widehat  B_0/\eps)\hat{w}(t )/\eps, \qquad  \qquad x (0)=x_0,\\
\frac{\textmd{d}}{\textmd{d} t } \hat{w}(t )=&\exp(- t \widehat  B_0/\eps)F\big( x (t ),\exp( t \widehat  B_0/\eps)\hat{w}(t )\big)\\
=&\exp(- t \widehat  B_0/\eps)\frac{\widehat  B( x (t ))-\widehat  B_0}{\eps } \exp( t \widehat  B_0/\eps)\hat{w}(t )\\
& +\eps \exp(- t \widehat  B_0/\eps)E( x (t )),\ \ \ \hat{w}(0)=\eps v_0.\end{aligned}
\end{equation}
 We first take the inner product on both sides of \eqref{H1 model problem} with
$ x (t ), \hat{w}(t )$ and then use Cauchy-Schwarz inequality to get
\begin{equation}\label{NEE}
\begin{aligned}
&\frac{\textmd{d}}{\textmd{d} t } \norm{ x (t )}^2\leq 2\norm{ x (t )}\norm{\hat{w}(t )}/\eps,\\
& \frac{\textmd{d}}{\textmd{d} t } \norm{\hat{w}(t )}^2\leq 2\norm{\hat{w}^{\intercal}(t )\exp(- t \widehat  B_0/\eps)\frac{\widehat  B( x (t ))-\widehat  B_0}{\eps } \exp( t \widehat  B_0/\eps)\hat{w}(t )}+2 \eps   \norm{E( x (t ))} \norm{\hat{w}(t )}.\end{aligned}\end{equation}
Observing  the skew-symmetric
matrix $\widehat  B_0$, it is arrived  that  (see \cite{lubich19})
$$\exp(-t \widehat  B_0)=I_3-\frac{\sin(t  b)}{b}\widehat  B_0+\frac{1-\cos(t  b)}{b^2}\widehat  B_0^2$$
with $b=\norm{B( x_0)}$. Therefore, it is clear that
$$\big(\exp(-t \widehat  B_0)\big)^{\intercal}=I_3+\frac{\sin(t  b)}{b}\widehat  B_0+\frac{1-\cos(t  b)}{b^2}\widehat  B_0^2=\exp(-t \widehat  B_0),$$
which yields that
\begin{equation*}
\begin{aligned}&\hat{w}^{\intercal}(t )\exp(- t \widehat  B_0/\eps)\frac{\widehat  B( x (t ))-\widehat  B_0}{\eps } \exp( t \widehat  B_0/\eps)\hat{w}(t )\\
=&\Big(\hat{w}^{\intercal}(t )\exp(- t \widehat  B_0/\eps)\frac{\widehat  B( x (t ))-\widehat  B_0}{\eps } \exp( t \widehat  B_0/\eps)\hat{w}(t )\Big)^{\intercal}\\
=&-\hat{w}^{\intercal}(t )\exp(- t \widehat  B_0/\eps)\frac{\widehat  B( x (t ))-\widehat  B_0}{\eps } \exp( t \widehat  B_0/\eps)\hat{w}(t )=0.\end{aligned}\end{equation*}
Thence, the second inequality in \eqref{NEE}  becomes
\begin{equation}\label{a-priori dbound}
\begin{aligned}
\frac{\textmd{d}}{\textmd{d} t } \norm{\hat{w}(t )} &\leq  { \eps  \norm{E( x (t ))}\leq\eps   \norm{E( x (0))}+  \eps   \norm{E( x (t ))-E( x (0))} }\\
& \leq   \eps   \norm{E( x (0))}  +  \eps   C\norm{  x (t )} +  \eps   C\norm{ x (0)}.\end{aligned}\end{equation}
Combining this with the first inequality in \eqref{NEE} and considering $x (0)=x_0, \hat{w}(0)=\eps v_0$ yields
 \begin{equation}\label{a-priori bound 1}
  \norm{ x (t )} \leq C,\ \ \  \norm{\hat{w}(t )} \leq C \eps \ \ \textmd{for\ all}\  \ \ t \in[0,T],
\end{equation}
by Gronwall’s inequality. This completes the proof of
   \eqref{lem1 bound1}. 

{
By using the derived bounds in \eqref{lem1 bound1} and the scheme of $F(x,w)$, the
 estimate \eqref{lem1 bound2} is shown as
$$\norm{F( x ,w )}\leq  \norm{\widehat  B( x )-\widehat  B_0}  \norm{w }/\eps  +\eps \norm{E( x )}\leq C +C \eps \leq C.$$
For the case that $B=B(\eps x)$, it is  obtained that
$$\norm{  B(\eps x  )-  B_0} \leq  \norm{\eps \nabla   B\big(\eps (x +\theta (x_0-x )) \big) }  \norm{x -x_0 }   \leq C \eps \  \ \textmd{with\ some}\ \ \theta\in[0,1],$$
which yields \eqref{lem1 bound3} immediately.}
The proof is complete.
\end{proof}

The following lemma is devoted to the two-scale system \eqref{2scalenew}  and derives the estimates of
its solution as well as derivatives w.r.t. $t $.  In this part, we introduce  $L_{t }^{\infty}:=L^{\infty}([0,T])$
and $\mathcal{H}^{m_0}:=\mathcal{H}^{m_0}([-\pi,\pi])$ to denote the functional spaces of $t $ and $\tau$ variables, respectively.

\begin{lem}\label{lem2}  
Under the conditions  given in Theorem \ref{UA thm2}, we have the following results.

\textbf{(i)}
 The two-scale system \eqref{2scale compact}  with the initial value  $U_0(\tau)$ \eqref{inv}  has a unique solution   $U(t ,\tau):=(Q^\intercal(t ,\tau),P^\intercal(t ,\tau))^\intercal$ in  $\mathcal{H}^{m_0}.$
 Moreover, the solution  satisfies
\begin{equation}\label{UBOUND1}\begin{aligned} &\normss{Q(t ,\tau)}\leq C,\ \ \normss{P(t ,\tau)}\leq C \eps,\\
& \normss{F\big( Q(t ,\tau)+s_1(\tau B_0)P(t ,\tau),s_0(\tau B_0)P(t ,\tau)\big) }\leq C,\   (t ,\tau)\in[0,T]\times
[-\pi,\pi]. \end{aligned}\end{equation}

\textbf{(ii)} The solution $U(t ,\tau)$ has up to four  derivatives w.r.t. $t $ which are functions in $\mathcal{H}^{m_0}$ and are  estimated as
\begin{equation}\label{UBOUND2}\begin{aligned}
&\normss{\partial_{t }^{k}U(t ,\tau)}\leq C,\ \  (t ,\tau)\in[0,T]\times
[-\pi,\pi], ,\ \ k=1,2,3,4.\end{aligned}\end{equation}

\textbf{(iii)}  { If the magnetic field has the maximal ordering form $B(\eps x)/\eps$, the following improved bounds are  obtained
\begin{equation}\label{IUBOUND}\begin{aligned}
& \normss{F\big( Q(t ,\tau)+s_1(\tau B_0)P(t ,\tau),s_0(\tau B_0)P(t ,\tau)\big) }\leq C \eps,\\
& \normss{\partial_{t }^{k}U(t ,\tau)}\leq C \eps^{k},\ \   (t ,\tau)\in[0,T]\times
[-\pi,\pi],\ \ k=1,2,3,4. \end{aligned}\end{equation}}
\end{lem}

\begin{proof} \textbf{Proof of (i).}
Firstly, combine  the expression \eqref{ftau} and    Assumption I proposed in Theorem \ref{UA thm2} with the results  \eqref{kapa re2} of $\kappa_1,\kappa_2,\kappa_3$. Then noticing the fact that $\Pi,\ A$ are bounded on $C^0(\mathbb{T};\mathcal{H}^{\sigma})$ for any $\sigma$, we get that the initial value $U_0(\tau)=\big(Q^\intercal(0 ,\tau),P^\intercal(0 ,\tau)\big)^\intercal$ determined by \eqref{inv} is  bounded in $\mathcal{H}^{m_0}$ as
\begin{equation}\label{ini0BOUND1}\begin{aligned} &\norms{Q(0 ,\tau)}\leq C,\ \ \norms{P(0 ,\tau)}\leq C \eps. \end{aligned}\end{equation}

  To show the existence and uniqueness of the solution $U(t ,\tau)$,  we introduce
  a new function $\chi(t ,\tau)=U(t ,\tau+t /\eps ).$ In the light of
\eqref{2scale compact}, it is known that
$\partial_{t }\chi(t ,\tau)= f_{\tau+t /\eps }(\chi(t ,\tau))$ with $\chi(0,\tau)=U_0(\tau) \in \mathcal{H}^{m_0}.$
Using the Cauchy–Lipschitz theorem in  $\mathcal{H}^{m_0}$ leads to the existence and uniqueness of the solution $U(t ,\tau)$ in $\mathcal{H}^{m_0}$.

To derive different bounds on $Q(t ,\tau)$  and $P(t ,\tau)$, we transform them  into two new functions $$\mathcal{Q}(t ,\tau):=Q(t ,\tau+t / \eps )+s_0\big( (\tau+t / \eps) \widehat  B_0 \big)  P(t ,\tau+t / \eps ),\  \mathcal{P}(t ,\tau):=  P(t ,\tau+t / \eps ).$$ The two-scale system \eqref{2scalenew} can be transformed to
\begin{equation}\label{2scalenew3}\left\{\begin{split}
  &\partial_{t }\mathcal{Q}(t ,\tau)=\exp\big((\tau+t /\eps  )\widehat B_0\big)\mathcal{P}(t ,\tau)/\eps,\\
    &\partial_{t }\mathcal{P}(t ,\tau)=\exp\big(-(\tau+t /\eps  )\widehat B_0\big)F\big( \mathcal{Q}(t ,\tau), \exp\big((\tau+t /\eps  )\widehat B_0\big)\mathcal{P}(t ,\tau)\big),
    \end{split}\right.
\end{equation} with $\mathcal{Q}(0,\tau)=Q(0,\tau)$ and $\mathcal{P}(0,\tau)=P(0,\tau).$
By the inner product on both sides of \eqref{2scalenew3} with
$\mathcal{Q}(t ,\tau), \mathcal{P}(t ,\tau)$ and  Cauchy-Schwarz inequality, we have
\begin{equation*}
\begin{aligned}
\partial_{t } \norm{\mathcal{Q}(t ,\tau)}^2\leq & 2\norm{\mathcal{Q}(t ,\tau)}\norm{\mathcal{P}(t ,\tau)}/\eps,\\
 \partial_{t } \norm{\mathcal{P}(t ,\tau)}^2\leq& 2\norm{\mathcal{P}^{\intercal}(t ,\tau) \exp\big(-(\tau+t /\eps  )\widehat B_0\big)\frac{\widehat  B(\mathcal{Q}(t ,\tau))-\widehat  B_0}{\eps } \exp\big((\tau+t /\eps  )\widehat B_0\big)\mathcal{P}(t ,\tau)}\\
&+2 \eps   \norm{E(\mathcal{Q}(t ,\tau))} \norm{\mathcal{P}(t ,\tau)}.\end{aligned}\end{equation*}
This system with the initial value \eqref{ini0BOUND1} has a same expression as \eqref{NEE}. Hence  according to  the same arguments of Lemma \ref{lem1},
one gets $\norm{\mathcal{Q}(t ,\tau)}\leq C,  \norm{\mathcal{P}(t ,\tau)}\leq C \eps$,  and so that
$$\normss{Q(t ,\tau)}\leq C, \ \  \normss{P(t ,\tau)}\leq C \eps.$$
The other  result  given in \eqref{UBOUND1} can be easily  proved by the same arguments as Lemma \ref{lem1}.

 \textbf{Proof of  (ii).}
 In what follows, we study the  derivatives of $U$ w.r.t. $t $, which starts with $\partial_{t } U(t ,\tau):= V(t ,\tau)$. Differentiating w.r.t. $t $ on both sides of \eqref{2scale compact}
 leads to   \begin{equation}\label{VE}\partial_{t } V(t ,\tau)+\frac{1}{\eps  }
 \partial_\tau V(t ,\tau)=\partial_{U}  f_\tau(U(t ,\tau))V(t ,\tau).\end{equation}
 Its initial data is derived by
\begin{equation}\label{V0o}
\begin{aligned}
&V_0(\tau):=V(0,\tau)=\partial_{t } U(0,\tau)= f_\tau(U_0(\tau))-\frac{1}{\eps }L
U_0(\tau) \\
=& f_\tau(U_0(\tau))- L \kappa_1\big(\tau,\underline{U}^{[2]}(0)\big)-\eps   L\kappa_2\big(\tau,\underline{U}^{[1]}(0)\big)-\eps^{2 } L\kappa_3\big(\tau,\underline{U}^{[0]}(0)\big).\end{aligned}
.\end{equation}
In the light of the bounds \eqref{UBOUND1}, we get \begin{equation}\label{Vb1}f_\tau(U_0(\tau))=\mathcal{O}_{\mathcal{H}^{m_0}}(1)\end{equation} and then
the scheme \eqref{kapa re2} yields  \begin{equation}\label{Vb2}\kappa_1=\mathcal{O}_{\mathcal{H}^{m_0}}(1),\quad \kappa_2=\mathcal{O}_{\mathcal{H}^{m_0}}(1),\quad \kappa_3=\mathcal{O}_{\mathcal{H}^{m_0}}(1).\end{equation}
Here and from now on, we use the notation $\mathcal{O}_{\mathcal{H}^{m_0}}$ for the terms uniformly bounded in $\mathbb{T}$ with the appropriate $\mathcal{H}^{m_0}$-norm.
{ Considering \eqref{inv}, we have
\begin{equation}\label{U0M}
\begin{aligned} &U_0(\tau)-\underline{U}^{[2]}(0)\\
=&\underline{U}^{[3]}(0)-\underline{U}^{[2]}(0)+\eps \kappa_1\big(\tau,\underline{U}^{[2]}(0)\big)+\eps^{2 } \kappa_2\big(\tau,\underline{U}^{[1]}(0)\big)+\eps^{3 } \kappa_3\big(\tau,\underline{U}^{[0]}(0)\big)\\
=&U_0(0)- \eps  \kappa_1\big(0,\underline{U}^{[2]}(0)\big)- \eps^2  \kappa_2\big(0,\underline{U}^{[1]}(0)\big)- \eps^3  \kappa_3\big(0,\underline{U}^{[0]}(0)\big)-U_0(0)+ \eps  \kappa_1\big(0,\underline{U}^{[1]}(0)\big)\\
&+ \eps^2  \kappa_2\big(0,\underline{U}^{[0]}(0)\big) +\eps \kappa_1\big(\tau,\underline{U}^{[2]}(0)\big)+\eps^{2 } \kappa_2\big(\tau,\underline{U}^{[1]}(0)\big)+\eps^{3 } \kappa_3\big(\tau,\underline{U}^{[0]}(0)\big)\\
=& \eps  \kappa_1\big(0,\underline{U}^{[1]}(0)\big)- \eps  \kappa_1\big(0,\underline{U}^{[2]}(0)\big)+\eps^2  \kappa_2\big(0,\underline{U}^{[0]}(0)\big) - \eps^2  \kappa_2\big(0,\underline{U}^{[1]}(0)\big) \\
&- \eps^3  \kappa_3\big(0,\underline{U}^{[0]}(0)\big) + \eps \kappa_1\big(\tau,\underline{U}^{[2]}(0)\big)+\eps^{2 } \kappa_2\big(\tau,\underline{U}^{[1]}(0)\big)+\eps^{3 } \kappa_3\big(\tau,\underline{U}^{[0]}(0)\big)\\
=&\eps \kappa_1\big(\tau,\underline{U}^{[2]}(0)\big)+\mathcal{O}_{\mathcal{H}^{m_0}}(\eps^2)=\mathcal{O}_{\mathcal{H}^{m_0}}(\eps).\end{aligned}
\end{equation}}With the result of $\kappa_1$ given in \eqref{kapa re2} and by Taylor expansion, it is obtained that
  \[
\begin{aligned}
 & f_\tau(U_0(\tau))- L \kappa_1\big(\tau,\underline{U}^{[2]}(0)\big)\\
  =& f_\tau(U_0(\tau))- LAf_\tau(\underline{U}^{[2]}(0))=\Pi f_\tau(\underline{U}^{[2]}(0))+f_\tau(U_0(\tau))- f_\tau(\underline{U}^{[2]})\\
=&\Pi f_\tau(\underline{U}^{[2]}(0))+  \partial_{U} f_\tau(\underline{U}^{[2]}(0))  \big(U_0(\tau)-\underline{U}^{[2]}(0)\big)+\mathcal{O}_{\mathcal{H}^{m_0}}(\eps^{2})=\mathcal{O}_{\mathcal{H}^{m_0}}(1).\end{aligned}
\]
Considering the scheme of $\kappa_2$, it can be obtained that
  \[
\begin{aligned}
L \kappa_2\big(\tau,\underline{U}^{[1]}(0)\big)=(I-\Pi)\partial_Uf_\tau(\underline{U}^{[1]}(0))Af_\tau(\underline{U}^{[1]}(0))-A\partial_U f_\tau(\underline{U}^{[1]}(0))\Pi f_\tau(\underline{U}^{[1]}(0))\end{aligned}
\]
 and so that $L \kappa_2=\mathcal{O}_{\mathcal{H}^{m_0}}(1)$.  Similarly, one deduces that $L \kappa_3=\mathcal{O}_{\mathcal{H}^{m_0}}(1)$. Based on the above analysis, the initial data \eqref{V0o} is bounded by $V_0(\tau)=\mathcal{O}_{\mathcal{H}^{m_0}}(1)$.
 Therefore, using the bootstrap type argument and Gronwall lemma to the differential equation \eqref{VE} gives the  existence, uniqueness and
  estimate
$$ \normss{ V(t ,\tau)} \leq
 \normss{V_0(\tau)}  e^{C  t }\leq
 \normss{V_0(\tau)}  e^{C  T}\leq C  \ \ \textmd{for}\ \  t \in[0,T].$$

 For the second derivative  $\partial^2_{t } U(t ,\tau):= W(t ,\tau)$, it satisfies the following differential equation
  \begin{equation}\label{2df}
  \partial_{t } W(t ,\tau)+\frac{1}{\eps }
 \partial_\tau W(t ,\tau)=\partial^2_U  f_\tau(U(t ,\tau)) (V(t ,\tau),V(t ,\tau))+\partial_{U}  f_\tau(U(t ,\tau)) W(t ,\tau),\end{equation}
with the initial data
  \[
\begin{aligned}
W_0(\tau):=&W(0,\tau)=
\partial_{t } V(0,\tau)=  \partial_{U}  f_\tau(U_0(\tau))V_0(\tau)-\frac{1}{\eps }L V_0(\tau).\end{aligned}
\]
Inserting $V_0(\tau)$ \eqref{V0o} into this data and after some  calculations, we can express   $W_0(\tau)$ as
{ \[\begin{aligned}
W_0(\tau)
= &   \big(\partial_{U}  f_\tau(U_0(\tau))-\frac{1}{\eps }L\big)\Big(\Pi f_\tau(\underline{U}^{[2]}(0))+  \partial_{U} f_\tau(\underline{U}^{[2]}(0))  \big(U_0(\tau)-\underline{U}^{[2]}(0)\big)\\
& -\eps   L\kappa_2\big(\tau,\underline{U}^{[1]}(0)\big)-\eps^{2 } L\kappa_3\big(\tau,\underline{U}^{[0]}(0)\big)+\mathcal{O}_{\mathcal{H}^{m_0}}(\eps^{2})\Big)\\
= &   \partial_{U}  f_\tau(U_0(\tau))\Pi f_\tau(\underline{U}^{[2]}(0))-\frac{1}{\eps }L\partial_{U} f_\tau(\underline{U}^{[2]}(0))  \big(U_0(\tau)-\underline{U}^{[2]}(0)\big)\\
&+L^2\kappa_2\big(\tau,\underline{U}^{[1]}(0)\big)+\mathcal{O}_{\mathcal{H}^{m_0}}(\eps).\end{aligned}
\]
In the light of this result and \eqref{U0M},  it is deduced that the dominating  part of $W_0(\tau)$ is
              \begin{equation*}
\begin{aligned}
 & \partial_{U}  f_\tau(U_0(\tau))\Pi f_\tau(\underline{U}^{[2]}(0))- L\partial_{U} f_\tau(\underline{U}^{[2]}(0))  \kappa_1\big(\tau,\underline{U}^{[2]}(0)\big)+L^2\kappa_2\big(\tau,\underline{U}^{[1]}(0)\big)\\
  =& \partial_{U}  f_\tau(U_0(\tau))\Pi f_\tau(\underline{U}^{[2]}(0))- L\partial_{U} f_\tau(\underline{U}^{[2]}(0))  \kappa_1\big(\tau,\underline{U}^{[2]}(0)\big)+L^2\kappa_2\big(\tau,\underline{U}^{[2]}(0)\big)+\mathcal{R}_1\\
    =& \partial_{U}  f_\tau(U_0(\tau))\Pi f_\tau(\underline{U}^{[2]}(0))- L\partial_{U} f_\tau(\underline{U}^{[2]}(0))  Af_\tau(\underline{U}^{[2]}(0))  \\&+L^2A\partial_{U} f_\tau(\underline{U}^{[2]}(0))  Af_\tau(\underline{U}^{[2]}(0)) -L^2A^2 \partial_{U}  f_\tau(\underline{U}^{[2]}(0))\Pi f_\tau(\underline{U}^{[2]}(0))+\mathcal{R}_1\\
        =& \partial_{U}  f_\tau(U_0(\tau))\Pi f_\tau(\underline{U}^{[2]}(0))- L\partial_{U} f_\tau(\underline{U}^{[2]}(0))  Af_\tau(\underline{U}^{[2]}(0))  \\&+L \partial_{U} f_\tau(\underline{U}^{[2]}(0))  Af_\tau(\underline{U}^{[2]}(0)) -(I-\Pi) \partial_{U}  f_\tau(\underline{U}^{[2]}(0))\Pi f_\tau(\underline{U}^{[2]}(0))+\mathcal{R}_1\\
                =& \partial_{U}  f_\tau(U_0(\tau))\Pi f_\tau(\underline{U}^{[2]}(0)) - \partial_{U}  f_\tau(\underline{U}^{[2]}(0))\Pi f_\tau(\underline{U}^{[2]}(0)) +\Pi \partial_{U}  f_\tau(U_0)\Pi f_\tau(\underline{U}^{[2]}(0))+\mathcal{R}_1\\
   = &\Pi \partial_{U}  f_\tau(U_0(\tau))\Pi f_\tau(\underline{U}^{[2]}(0))+ \mathcal{R}_1+ \mathcal{R}_2,
\end{aligned}
  \end{equation*}
  with the remainders
        \begin{equation*}
\begin{aligned}
 \mathcal{R}_1=&L^2\kappa_2\big(\tau,\underline{U}^{[1]}(0)\big)- L^2\kappa_2(\tau,\underline{U}^{[2]}(0)),\ \
  \mathcal{R}_2= \big(\partial_{U}  f_\tau(U_0(\tau)) - \partial_{U}  f_\tau(\underline{U}^{[2]}(0))\big)\Pi f_\tau(\underline{U}^{[2]}(0)),
\end{aligned}
  \end{equation*}
where we have used  $LA=I-\Pi$ and $L^2A=L$.
Utilizing the estimates  \eqref{U0M} and
        \begin{equation*}
\begin{aligned}
\underline{U}^{[1]}(0)-\underline{U}^{[2]}(0)=& \eps  \kappa_1\big(0,\underline{U}^{[1]}(0)\big)+\eps^2  \kappa_2\big(0,\underline{U}^{[0]}(0)\big)-\eps  \kappa_1\big(0,\underline{U}^{[0]}(0)\big)=\mathcal{O}_{\mathcal{H}^{m_0}}(\eps^2),
\end{aligned}
  \end{equation*}
  we have
 $ \mathcal{R}_2=\mathcal{O}_{\mathcal{H}^{m_0}}(\eps^2)$ and    $ \mathcal{R}_1=\mathcal{O}_{\mathcal{H}^{m_0}}(\eps)$, respectively.
Considering these results and \eqref{Vb1}-\eqref{Vb2}, we deduce that $W_0(\tau)=\mathcal{O}_{\mathcal{H}^{m_0}}(1 )$ and $LW_0(\tau)=L\mathcal{R}_1+ L\mathcal{R}_2+\mathcal{O}_{\mathcal{H}^{m_0}}(\eps)=\mathcal{O}_{\mathcal{H}^{m_0}}(\eps)$ (this estimate is needed for deriving the bound of $\partial^3_{t} U(t ,\tau)$).}
Using Gronwall lemma to   \eqref{2df}  again, we have
$$ \normss{ W(t ,\tau)} \leq
\big(\normss{W_0(\tau)}+C t   \big)  e^{C   t }\leq C   \ \ \ \textmd{for}\ \  t \in[0,T].$$

This procedure can be proceeded in an analogous way for $\partial^3_{t } U(t ,\tau)$ and $\partial^4_{t } U(t ,\tau)$, and the statements \eqref{UBOUND2} for $k=3,4$ are obtained.

 {  \textbf{Proof of (iii).}
In  the case of maximal ordering magnetic field,  using
the form $B(\eps x)$, one gets
 $$\normss{\widehat B\big(\eps Q(t ,\tau)+\eps s_1(\tau\widehat  B_0)P(t ,\tau)\big)-\widehat  B_0}\leq C \eps.$$
This bound and the second  result of \eqref{UBOUND1}  immediately yield
\[
\begin{aligned}  &\normss{F\big( Q(t ,\tau)+s_1(\tau\widehat B_0)P(t ,\tau),s_0(\tau\widehat B_0)P(t ,\tau)\big)}\\
\leq & \normss{\frac{\widehat  B\big(\eps Q(t ,\tau)+\eps s_1(\tau\widehat B_0)P(t ,\tau)\big)-\widehat   B_0}{\eps }s_0(\tau \widehat B_0)P(t ,\tau)}  \\ &+\eps \normss{E\big(Q(t ,\tau)+ s_1(\tau\widehat B_0)P(t ,\tau)\big)}\leq C\eps.
\end{aligned}
\]
This shows the first statement of \eqref{IUBOUND}.
For the second estimate of \eqref{IUBOUND}, we only present the proof for $k=1$ for brevity.
From the analysis given in the proof of  (ii), it is easy to derive the initial value $V_0$ of \eqref{VE}, i.e.,
\[\begin{aligned}
V_0(\tau)
= &   \Pi f_\tau(\underline{U}^{[2]}(0))+  \partial_{U} f_\tau(\underline{U}^{[2]}(0)) \big (U_0(\tau)-\underline{U}^{[2]}(0)\big) \\
&-\eps   L\kappa_2\big(\tau,\underline{U}^{[1]}(0)\big)-\eps^{2 } L\kappa_3\big(\tau,\underline{U}^{[0]}(0)\big)+\mathcal{O}_{\mathcal{H}^{m_0}}(\eps^{2}).\end{aligned}
\]
Considering  $f_\tau(\underline{U}^{[2]}(0))=\mathcal{O}_{\mathcal{H}^{m_0}}(\eps)$ and $\partial_{U} f_\tau(\underline{U}^{[2]}(0)) \big (U_0(\tau)-\underline{U}^{[2]}(0)\big)=\mathcal{O}_{\mathcal{H}^{m_0}}(\eps)$, it is deduced that
$V_0(\tau)=\mathcal{O}_{\mathcal{H}^{m_0}}(\eps).$
Then using  Gronwall lemma to  \eqref{VE}, we get  $$\normss{\partial_{t } U(t ,\tau)}\leq C \eps.$$
The estimates for $\partial^k_{t } U(t ,\tau)$ with $k=2,3,4$ can be derived with the same arguments.

The proof is complete. }
\end{proof}

So far, the analysis of the transformed systems \eqref{charged-particle2 ccc} and \eqref{2scale compact} has been finished. In the rest part of this section, we are devoted to  the convergence of fully discrete scheme.  We start with the error brought by the Fourier pseudospectral method in the variable $\tau\in\bT:=[-\pi,\pi]$.
For this purpose,  a transitional integrator will be introduced    and then the estimations for  the original integrator  can be turned to the estimations for this transitional integrator.
To this end, choose an even positive integer $N_\tau$ for the set $\mathcal{M}: =\{-N_\tau/2,-N_\tau/2+1,\ldots,N_\tau/2-1\}$ and define the standard projection operator  $P_{\mathcal{M}} : L^2([-\pi,\pi]) \rightarrow
Y_{\mathcal{M}}:=\textmd{span}\{\exp( \mathrm{i}   k \tau ),\ k\in \mathcal{M},\ \tau \in[-\pi,\pi]\}
$ as
$$(P_{\mathcal{M}}\vartheta)(\tau)=\sum\limits_{k\in \mathcal{M}} \widetilde{\vartheta}_k \exp( \mathrm{i}k \tau).$$
Here $ \widetilde{\vartheta}_k$ are the Fourier   transform coefficients of the periodic function $\vartheta(\tau)$ which are defined as
$$ \widetilde{\vartheta}_k=\frac{1}{2\pi}\int_{-\pi}^{\pi}\vartheta(\tau)\exp(-\mathrm{i}k\tau)d\tau,\ \ \  k\in \mathcal{M}.$$
The transitional integrator (TI)  of the two-scale system \eqref{2scale compact} is given by
\begin{equation}\label{TIM}\begin{array}[c]{ll}%
&\mathcal{U}_j^{ni}:=\sum\limits_{k\in \mathcal{M}} \widetilde{\mathbf{U}}^{ni}_{k,j}
\exp(\mathrm{i}j \tau )\approx P_{\mathcal{M}} U_j\big((n +c_i)h,\tau\big),\ \ i=1,2,\ldots,s,\\
&\mathcal{U}_j^{n}=\sum\limits_{k\in \mathcal{M}}\widetilde{\mathbf{U}}^{n}_{k,j}
\exp(\mathrm{i}j \tau )\approx P_{\mathcal{M}} U_j(nh, \tau),\ \ j=1,2,\ldots,6,\ \ n=1,2,\ldots,T/h,\end{array}\end{equation}
where   $s$-stage explicit exponential integrators are used  for computing $ \widetilde{\mathbf{U}}^{ni}$ and $ \widetilde{\mathbf{U}}^{n}$:  \begin{equation}\label{ei-sc-f}
\begin{array}[c]{ll}%
\widetilde{\mathbf{U}}^{ni}&=\exp(c_{i}hM) \widetilde{\mathbf{U}}^{n}+h\textstyle\sum\limits_{j=1}^{i-1}\bar{a}_{ij}(hM)\widetilde{\mathbf{F}}( \mathcal{U}^{ni}),\ \
i=1,2,\ldots,s,\\
\widetilde{\mathbf{U}}^{n+1}&=\exp( hM)  \widetilde{\mathbf{U}}^{n}+h\textstyle\sum\limits_{j=1}^{s}
\bar{b}_{j}(hM)\widetilde{\mathbf{F}}( \mathcal{U}^{nj}),\ \
 \quad n=0,1,\ldots,T/h-1,
\end{array}\end{equation}
with   the   Fourier   coefficient  $\widetilde{\mathbf{F}}$ of $\mathbf{F}$.
The started value $\widetilde{\mathbf{U}}^{0}$ is given by $$\widetilde{\mathbf{U}}_k^{0}:=\frac{1}{2\pi}\int_{-\pi}^{\pi}U_0(\tau)\exp(-\mathrm{i}k\tau)d\tau,\ \ \  k\in \mathcal{M},$$ with $U_0(\tau)$ given in \eqref{inv}. 

\begin{lem}\label{lem3} For solving the two-scale system \eqref{2scale compact}, define the projected errors of   the transitional integrator \eqref{TIM}--\eqref{ei-sc-f} as
$$e^{n}_{\mathrm{proj}}=\mathcal{U}^{n}-P_{\mathcal{M}} U(nh, \tau),\ \ \ e^{ni}_{\mathrm{proj}}=\mathcal{U}^{ni}-P_{\mathcal{M}} U\big((n +c_i)h,\tau\big),$$
where $n=1,2,\ldots,T/h$ and $i=1,2,\ldots,s.$
Introduce another operator  $I_{\mathcal{M}} : C([-\pi,\pi]) \rightarrow
Y_{\mathcal{M}}
$ which is defined as
$$(I_{\mathcal{M}}\vartheta)(\tau)=\sum\limits_{k\in \mathcal{M}}\widehat{\vartheta}_k e^{ \mathrm{i}k \tau } \  \textmd{with\ discrete\ Fourier\  coefficients}\ \widehat{\vartheta}_k.$$
The errors between  the original scheme \eqref{ei-sc}-\eqref{nsfor2scale} and the two-scale system \eqref{2scale compact}  are denoted by
$$e^{n}=I_{\mathcal{M}}\mathbf{U}^{n}-  U(nh, \tau ),\ \ \ e^{ni}=I_{\mathcal{M}}\mathbf{U}^{ni}-  U\big((n +c_i)h, \tau\big).$$ Under the conditions of Assumption II given in Theorem \ref{UA thm2},
the above two errors satisfy
\begin{equation}\label{err rea}
\begin{array}[c]{ll}%
 \norms{e^{n}}\leq \norms{e^{n}_{\mathrm{proj}}}+ C(2\pi/N_{\tau})^{m_0},\ \  \norms{e^{ni}}\leq \norms{e^{ni}_{\mathrm{proj}}}+ C(2\pi/N_{\tau})^{m_0}.
\end{array}\end{equation}
\end{lem}
\begin{proof}
By using the
triangle inequality and the estimates on projection error \cite{Shen}, we immediately get that
\[
\begin{aligned}  &\norms{I_{\mathcal{M}}\mathbf{U}^{n}-  U(nh,\tau)}
\leq  \norms{\mathcal{U}^{n}-P_{\mathcal{M}} U(nh, \tau)}+ \norms{I_{\mathcal{M}}\mathbf{U}^{n}-\mathcal{U}^{n}}\\&\qquad+ \norms{P_{\mathcal{M}} U(nh,\tau)-  U(nh,\tau)}
 \leq  \norms{e^{n}_{\mathrm{proj}}}+ C(2\pi/N_{\tau})^{m_0}.
\end{aligned}
\]
 The second statement can be shown in a same way.
\end{proof}

By this lemma, it follows that  the results for $e^n, e^{ni}$   can be turned to  the estimations for  $e^{n}_{\mathrm{proj}},\ e^{ni}_{\mathrm{proj}}$. 
In what follows, we study these error bounds on  the transitional integrator  \eqref{TIM}--\eqref{ei-sc-f}, whose
  error system   is presented as
$$(e^{ni}_{\mathrm{proj}})_{j}=\sum\limits_{k\in \mathcal{M}}
E_{j,k}^{ni} \exp(\mathrm{i} k  \tau),\ \
(e^{n}_{\mathrm{proj}})_{j}=\sum\limits_{k\in \mathcal{M}}
E_{j,k}^{n} \exp(\mathrm{i} k  \tau),$$
where the error equation reads
\begin{equation}\label{error equ}
\begin{aligned}
E^{ni}&=\exp( c_ihM)E^{n}+h \sum\limits_{j=1}^{i-1}\bar{a}_{i,j}(hM) \Delta \mathbf{F} ^{nj} + \delta^{ni},\\
E^{n+1}&=\exp( hM)E^{n}+h \sum\limits_{j=1}^{s}\bar{b}_{j}(hM) \Delta \mathbf{F} ^{nj} + \delta^{n+1},
\end{aligned}
\end{equation}
with
$\Delta \mathbf{F} ^{nj}=  \widetilde{\mathbf{F}}\big(P_{\mathcal{M}} U(nh+c_jh,\tau)\big)-\widetilde{\mathbf{F}}( \mathcal{U}^{nj}).$
 Here the remainders
$ \delta^{ni}$  and
$ \delta^{n+1}$ are determined by  inserting the exact solution $\widetilde{\mathbf{U}}(t ):= \big(\widetilde{U}_{k,j}(t ,\tau)\big)_{k\in \mathcal{M},\ j=1,2,\ldots,6}$
into the scheme \eqref{ei-sc-f}, {i.e.,}
\begin{equation}\label{remainders1}
\begin{aligned}
&\widetilde{\mathbf{U}}(nh+c_ih)=\exp( c_ihM) \widetilde{\mathbf{U}}(nh)+h\textstyle\sum\limits_{j=1}^{i-1}\bar{a}_{ij}(hM)\widetilde{\mathbf{F}}\big(P_{\mathcal{M}} U(nh+c_jh,\tau)\big)+\delta^{ni},\\
&\widetilde{\mathbf{U}}(nh+h)=\exp( hM) \widetilde{\mathbf{U}}(nh)+h\textstyle\sum\limits_{j=1}^{s}
\bar{b}_{j}(hM)\widetilde{\mathbf{F}}\big(P_{\mathcal{M}} U(nh+c_jh,\tau)\big)+\delta^{n+1}.
\end{aligned}
\end{equation}
The bounds of these remainders are derived in the following lemma.

\begin{lem} \label{lem5}
Suppose that the conditions  given in Theorem \ref{UA thm2} hold.
The remainders
$ \delta^{ni}$  and
$ \delta^{n+1}$ defined in \eqref{remainders1}  for $i=1,2,\ldots,s$ and $n=0,1,\ldots,T/\eps -1$ are bounded by
\begin{equation*}
\begin{aligned} &\textmd{MO1-E}:\quad \norms{  \delta^{ni}}\leq C    h,\quad\     \norms{  \delta^{n+1}}\leq C\big(h\norms{\psi_{1}(h  M) }+
h^{2}\big),\\
 &\textmd{MO2-E}:\quad \norms{  \delta^{ni}}\leq C    h^{2},\quad   \norms{  \delta^{n+1}}\leq C\big(h^2\norms{\psi_{2}(h  M) }+
h^{3}\big),\\
 &\textmd{MO3-E}:\quad \norms{  \delta^{ni}}\leq C
h^{3},\quad   \norms{  \delta^{n+1}}\leq C \big(h^3\norms{\psi_{3}(h  M) }+
h^{4}\big),\\
 &\textmd{MO4-E}:\quad \norms{  \delta^{ni}}\leq C
h^{4},\quad   \norms{  \delta^{n+1}}\leq C \big(h^4\norms{\psi_{4}(h  M) }+
h^{5}\big),\end{aligned} 
\end{equation*}
{ and for the maximal ordering magnetic field, they are improved to be
\begin{equation*}
\begin{aligned} &\textmd{MO1-E}:\quad \norms{  \delta^{ni}}\leq C \eps   h,\quad\ \    \norms{  \delta^{n+1}}\leq C\big(\eps
h \norms{\psi_{1}(h  M) }+ \eps^2
h^{2}\big),\\
 &\textmd{MO2-E}:\quad \norms{  \delta^{ni}}\leq C \eps^2   h^{2},\quad   \norms{  \delta^{n+1}}\leq C\big(\eps^2
h^{2}\norms{\psi_{2}(h  M) }+  \eps ^{3}
h^{3}\big),\\
 &\textmd{MO3-E}:\quad \norms{  \delta^{ni}}\leq C \eps ^{3}
h^{3},\quad   \norms{  \delta^{n+1}}\leq C \big(\eps^3
h^{3}\norms{\psi_{3}(h  M) }+ \eps ^{4}
h^{4}\big),\\
 &\textmd{MO4-E}:\quad \norms{  \delta^{ni}}\leq C \eps^{4}
h^{4},\quad   \norms{  \delta^{n+1}}\leq C \big(\eps^4
h^{4} \norms{\psi_{4}(h  M) }+ \eps ^{5}
h^{5}\big),\end{aligned} 
\end{equation*}}where $\psi_{\rho}(z):=\varphi_{\rho}(z)-\sum_{i=1}^s\bar{b}_i(z)\frac{c_i^{\rho-1}}{(\rho-1)!}$ for $\rho=1,2,3,4$.

\end{lem}
\begin{proof}
From the Duhamel principle, it follows  that the exact solution $\widetilde{\mathbf{U}}(t )$ satisfies
\begin{equation}\label{remainders2}
\begin{aligned}
&\widetilde{\mathbf{U}}(nh+c_ih)=\exp( c_ihM) \widetilde{\mathbf{U}}(nh)+\int_0^{c_ih}\exp((\theta-c_ih)M)\widetilde{\mathbf{F}}\big(P_{\mathcal{M}} U(nh+\theta,\tau)\big)d\theta,\\
&\widetilde{\mathbf{U}}(nh+h)=\exp( hM) \widetilde{\mathbf{U}}(nh)+\int_0^{h}\exp((\theta-h)M)\widetilde{\mathbf{F}}\big(P_{\mathcal{M}} U(nh+\theta,\tau)\big)d\theta.
\end{aligned}
\end{equation}
Subtracting \eqref{remainders2} from \eqref{remainders1}, one gets
 \begin{equation*}
\begin{aligned}
\delta^{ni}=&\int_0^{c_ih}\exp\big((\theta-c_ih)M\big)G(t _n+\theta)d\theta-h\textstyle\sum\limits_{j=1}^{i-1}\bar{a}_{ij}(hM)
G(t _n+c_jh), \\
\delta^{n+1}=&\int_0^{h}\exp\big((\theta-h)M\big)G(t _n+\theta)d\theta-h\textstyle\sum\limits_{j=1}^{s}
\bar{b}_{j}(hM)G(t _n+c_jh),
\end{aligned}
\end{equation*}
where we use the notation
$ G(t ):=\widetilde{\mathbf{F}}\big(P_{\mathcal{M}} U(t,\tau)\big)$.
Applying the Taylor expansion to the nonlinear function $G$ gives
\[
\begin{aligned} \delta^{n+1}=& h      \int_{0}^1   \exp\big((1-z)h  M\big)  \sum\limits_{\rho=1}^{r}\frac{(z  h  )^{\rho-1}}{(\rho-1)!}\frac{\textmd{d}^{\rho-1}}{\textmd{d} t ^{\rho-1}} G(t _n)   {\rm d}z\\&-h   \sum\limits_{j =1}^{s}   \bar{b}_{j}(h  M) \sum\limits_{\rho=1}^{r}\frac{ c_{j}^{\rho-1}h ^{\rho-1}}{(\rho-1)!}\frac{\textmd{d}^{\rho-1}}{\textmd{d} t ^{\rho-1}}G(t _n)+ \delta_{r}^{n+1}\\
=& \sum\limits_{\rho=1}^{r}h ^\rho \psi_{\rho}(h  M)  \frac{\textmd{d}^{\rho-1}}{\textmd{d} t ^{\rho-1}}G(t _n)+\delta^{n+1}_r,\end{aligned}
\]
and similarly $$
\delta^{ni}
= \sum\limits_{\rho=1}^{r-1}h ^\rho \psi_{\rho,i}(h  M)  \frac{\textmd{d}^{\rho-1}}{\textmd{d} t ^{\rho-1}}G(t _n)+\delta^{ni}_{r-1}.
$$
Here  we consider the notation
$
\psi_{\rho,i}(z)=\varphi_{\rho}(c_iz)c_i^{\rho}-\sum_{k=1}^{i-1}\bar{a}_{ik}(z)\frac{c_{k}^{\rho-1}}{(\rho-1)!}$
for $ i=1,2,\ldots,s,$ and $\delta^{n+1}_r, \delta^{ni}_{r-1}$ denote the difference brought by the truncation of   the summation in $\rho$.
With the help of the boundedness  derived in Lemma \ref{lem2} for  the solution of the two-scale system \eqref{2scale compact}, we know that
$\frac{\textmd{d}^{\rho-1}}{\textmd{d} t ^{\rho-1}}G(t _n)=\mathcal{O}_{\mathcal{H}^{m_0}}(1)$ and  for the maximal ordering magnetic field, $\frac{\textmd{d}^{\rho-1}}{\textmd{d} t ^{\rho-1}}G(t _n)=\mathcal{O}_{\mathcal{H}^{m_0}}(\eps^{\rho})$ with $\rho=1,2,\ldots,5$.
Based on which, it is deduced that  $$\delta^{ni}_{r-1}=\mathcal{O}_{\mathcal{H}^{m_0}}\big( h ^{r} \big),\ \ \ \delta^{n+1}_r=\mathcal{O}_{\mathcal{H}^{m_0}}\big( h ^{r+1} \big)$$
and for the maximal ordering  case $$\delta^{ni}_{r-1}=\mathcal{O}_{\mathcal{H}^{m_0}}\big(\eps^{r} h ^{r} \big),\ \ \ \delta^{n+1}_r=\mathcal{O}_{\mathcal{H}^{m_0}}\big(\eps^{r+1} h ^{r+1} \big).$$
Now we verify the results of $\psi_{\rho,i}(h  M)$ and $\psi_{\rho}(h  M)$ for those four proposed methods and then immediately get the conclusions  of this lemma.
\end{proof}

\textbf{Proof of Theorem \ref{UA thm2}.}

\begin{proof}In this part, we devote to the proof for the fourth-order integrator, MO4-E, as applied to the CPD under a maximal ordering magnetic field. The process of adapting this proof for lower-order methods or for a more general strong magnetic field is straightforward and we omit these modifications for the sake of conciseness.

Considering Taylor series for the  error recursion \eqref{error equ},  we can rewrite $ \Delta \mathbf{F} ^{nj}$ in a new form $ \Delta \mathbf{F} ^{nj}=J_nE^{nj}$ with a matrix $J_n$. Inserting this into the second formula of \eqref{error equ} and concerning the local error bounds stated  in Lemma \ref{lem5}, it is deduced that
 there exist bounded operators $\mathcal{N}^{ni}(E^{n})$ such that
$E^{ni}=\mathcal{N}^{ni}(E^{n})E^{n}
+\delta^{ni}+h ^4\eps^{4}\mathcal{R}^{ni}$ with uniformly bounded remainders $\mathcal{R}^{ni}$.
Then $E^{n+1}$ can be expressed by
\begin{equation*}
\begin{aligned}
E^{n+1}=&\exp( hM)E^{n}+h  \sum\limits_{\rho=1}^{5} \bar{b}_{\rho}(h  M) J_n\mathcal{N}^{n\rho}(E^{n})E^{n}\\
&+h  \sum\limits_{\rho=1}^{5} \bar{b}_{\rho}(h  M)J_n\delta^{n\rho} +   h ^4 \psi_{4}(h  M)\frac{\textmd{d}^{3}}{\textmd{d} t ^{3}}G(t _n)+\mathcal{O}_{\mathcal{H}^{m_0}}\big(h ^{5}\eps^{5 } \big)\\
=& \exp( hM)E^{n}+h   \mathcal{N}^{n}(E^{n}) E^{n} +   h ^4 \psi_{4}(h  M)\frac{\textmd{d}^{3}}{\textmd{d} t ^{3}}G(t _n)+\mathcal{O}_{\mathcal{H}^{m_0}}\big(h ^{5}\eps^{4} \big),
\end{aligned}\end{equation*}
with the notation $\mathcal{N}^{n}(E^{n}):=\sum\limits_{\rho=1}^{5} \bar{b}_{\rho}(h  M) J_n\mathcal{N}^{n\rho}(E^{n}).$
 This recursion can be solved as
\begin{equation*}
\begin{aligned}
E^{n}=&h   \sum_{j=0}^{n-1} \exp\big((n-j-1)h  M\big) \mathcal{N}^{j}(E^{j}) E^{j}\\
&+   h ^4\sum_{j=0}^{n-1} \exp\big((n-j-1)h  M\big)\psi_{4}(h  M)  \frac{\textmd{d}^{3}}{\textmd{d} t ^{3}}G(t _n)+\mathcal{O}_{\mathcal{H}^{m_0}}\big(h  ^{4}\eps^{4}\big).
\end{aligned}\end{equation*}
For the integrator MO4-E, we observe that it does not have  $\psi_{4}(hM)=0$ but satisfies $\psi_{4}(0)=0$. This weak condition can ensure that $\psi_{4}(h  M)=h  M \tilde{\psi}_{4}(h  M)$ for a bounded operator $\tilde{\psi}_{4}(h  M)$.
In the light of  Lemma 4.8 of \cite{Ostermann06}, one arrives
\begin{equation*}\begin{array}{ll}&\sum_{j=0}^{n-1} \exp\big((n-j-1)h  M\big)\psi_{4}(h  M) \frac{\textmd{d}^{3}}{\textmd{d} t ^{3}}G(t _n)\\
=&\sum_{j=0}^{n-1}\exp\big((n-j-1)h  M\big)h  M \tilde{\psi}_{4}(h  M)  \frac{\textmd{d}^{3}}{\textmd{d} t ^{3}}G(t _n)\\ =&\mathcal{O}_{\mathcal{H}^{m_0}} \big(\frac{\textmd{d}^{3}}{\textmd{d} t ^{3}}G(t _n)\big)
=\mathcal{O}_{\mathcal{H}^{m_0}} \big(\eps^{4 }\big).
\end{array}
\end{equation*}
This leads to
$E^{n}=h   \sum_{j=0}^{n-1} \exp\big((n-j-1)h  M\big) \mathcal{N}^{j}(E^{j}) E^{j}+\mathcal{O}_{\mathcal{H}^{m_0}}\big(h  ^{4}\eps^{4 }\big),$
 which contributes  $E^{n}=\mathcal{O}_{\mathcal{H}^{m_0}}\big(h  ^{4}\eps^{4 } \big).$
This statement and the formulation process given in Definition  \ref{dIUA-PE-F} immediately complete
the proof  of Theorem \ref{UA thm2} for MO4-E.
\end{proof}

\section{Numerical test}\label{sec:4}
In this section, we apply the proposed methods to a charged-particle dynamics under two different kinds of   strong magnetic fields.

 \textbf{Problem.}
Consider  the charged-particle dynamics \eqref{charged-particle 3d} under    the scalar potential $U(x)=\frac{1}{\sqrt{x_{1}^2+x_{2}^2}}$ { that has been presented in \cite{Hairer2017-2,lubich19} and}   determines $E(x)=-\nabla U(x)$.
  The initial values are chosen as $x(0)=(1/3, 1/4, 1/2)^{\intercal},  v(0)=(2/5, 2/3, 1)^{\intercal}$ and the problem is considered over $[0,1]$. { For this problem, we consider two different strong magnetic fields:
  \begin{equation} \textmd{the general strong magnetic field:}\ \ \ B(x)/\eps= \left(
                     \cos(  x_2),
                      1+\sin( x_3),
                    \cos(x_1)
                 \right)^{\intercal}/\eps,\label{gsm}%
\end{equation}
and
  \begin{equation} \textmd{the maximal ordering case:}\ \ \   B(x)/\eps= \left(
                     \cos(\eps   x_2) ,
                      1+\sin(\eps  x_3),
                    \cos(\eps  x_1)
                 \right)^{\intercal}/\eps.\label{moc}%
\end{equation}}For the discretization in the $\tau$-direction of the  methods,
we fix $N_\tau=2^6$ so that this part of error is rather negligible.
   \begin{figure}[t!]
$$\begin{array}{cc}
\psfig{figure=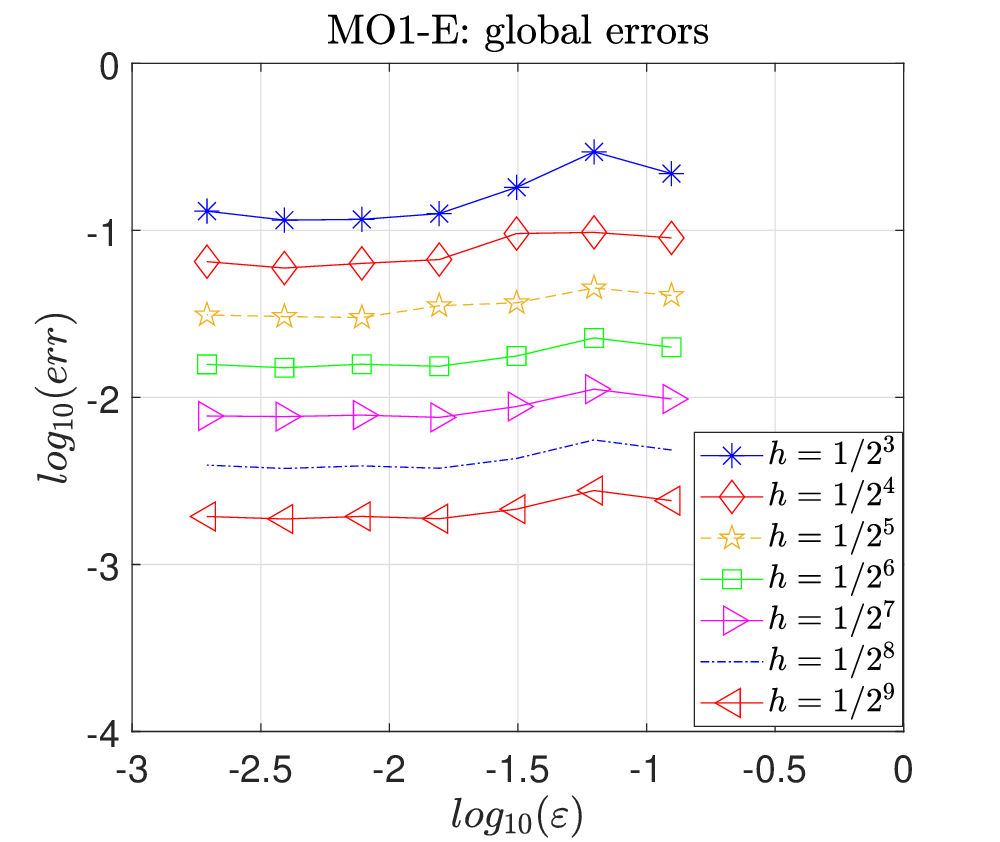,height=4.0cm,width=3.6cm}
\psfig{figure=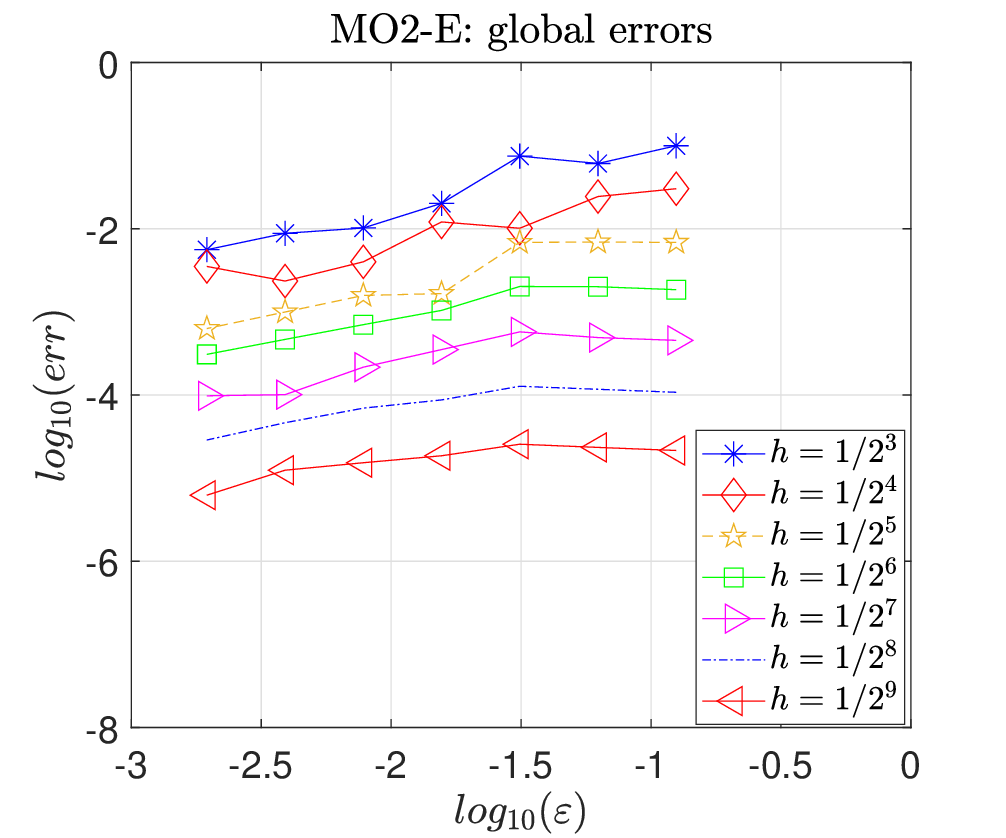,height=4.0cm,width=3.6cm}
\psfig{figure=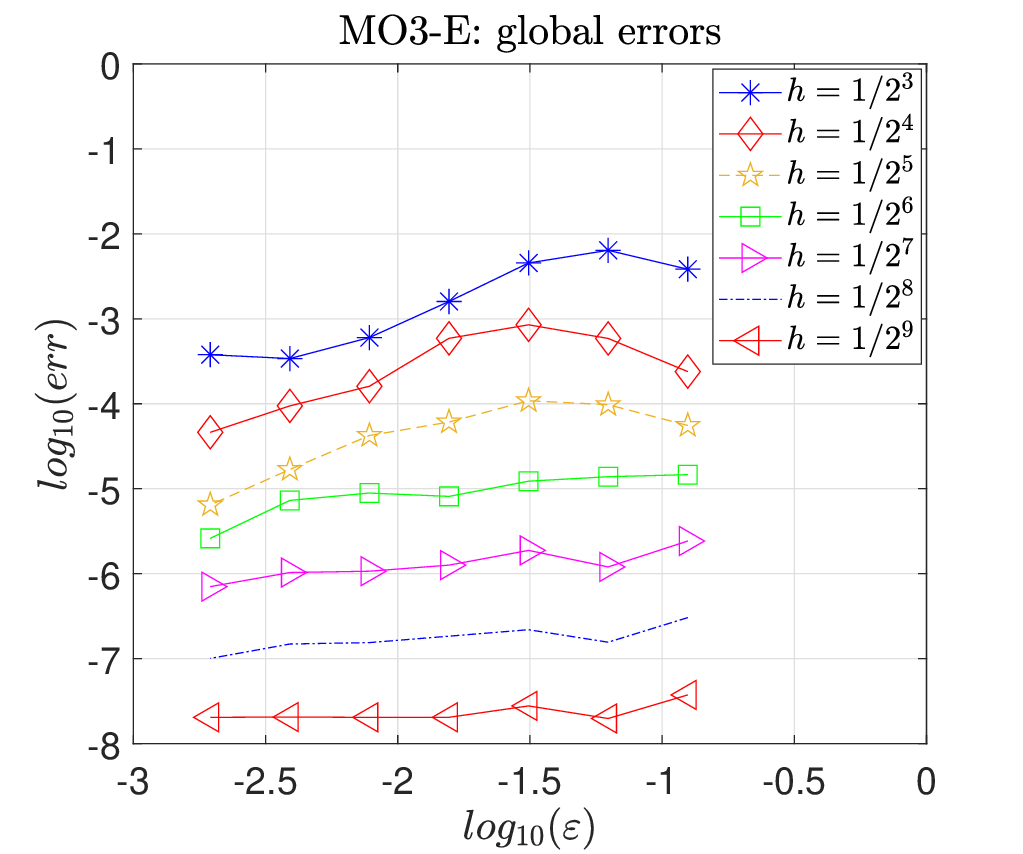,height=4.0cm,width=3.6cm}
\psfig{figure=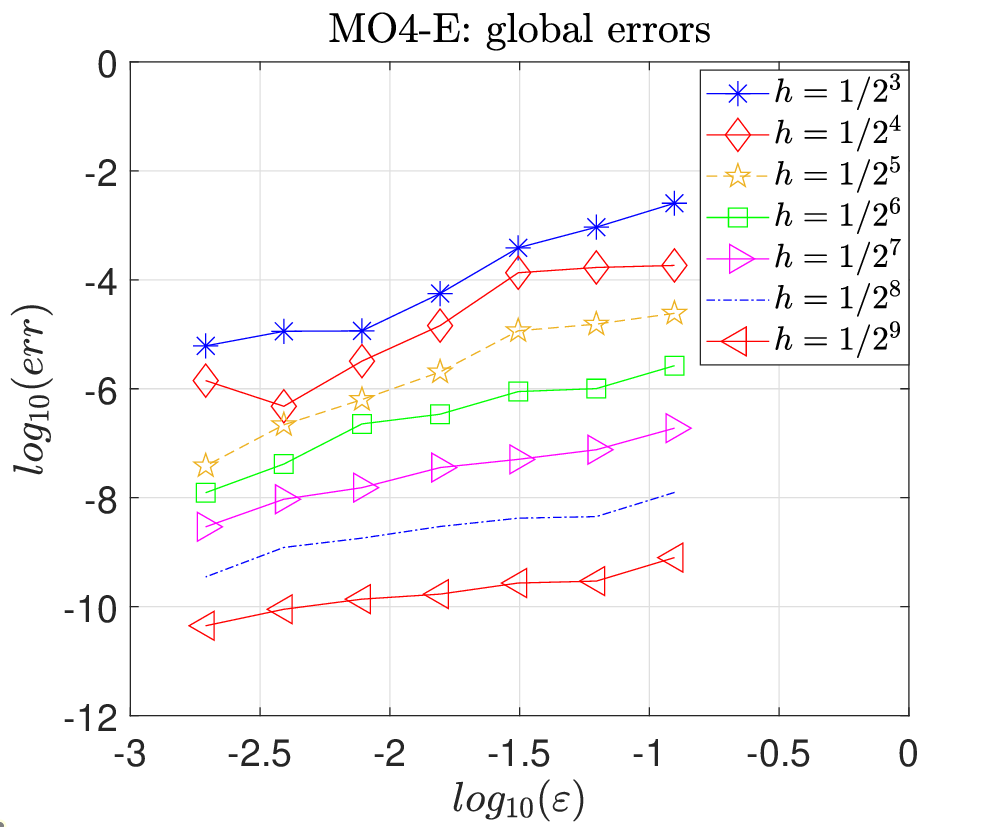,height=4.0cm,width=3.6cm}
\end{array}$$
\caption{Convergence order w.r.t. $\eps$ for the general strong magnetic field \eqref{gsm}:  the log-log plot of the temporal error   \eqref{err} { at $1$} against  $\eps$.}\label{fig1}
\end{figure}

   \begin{figure}[t!]
$$\begin{array}{cc}
\psfig{figure=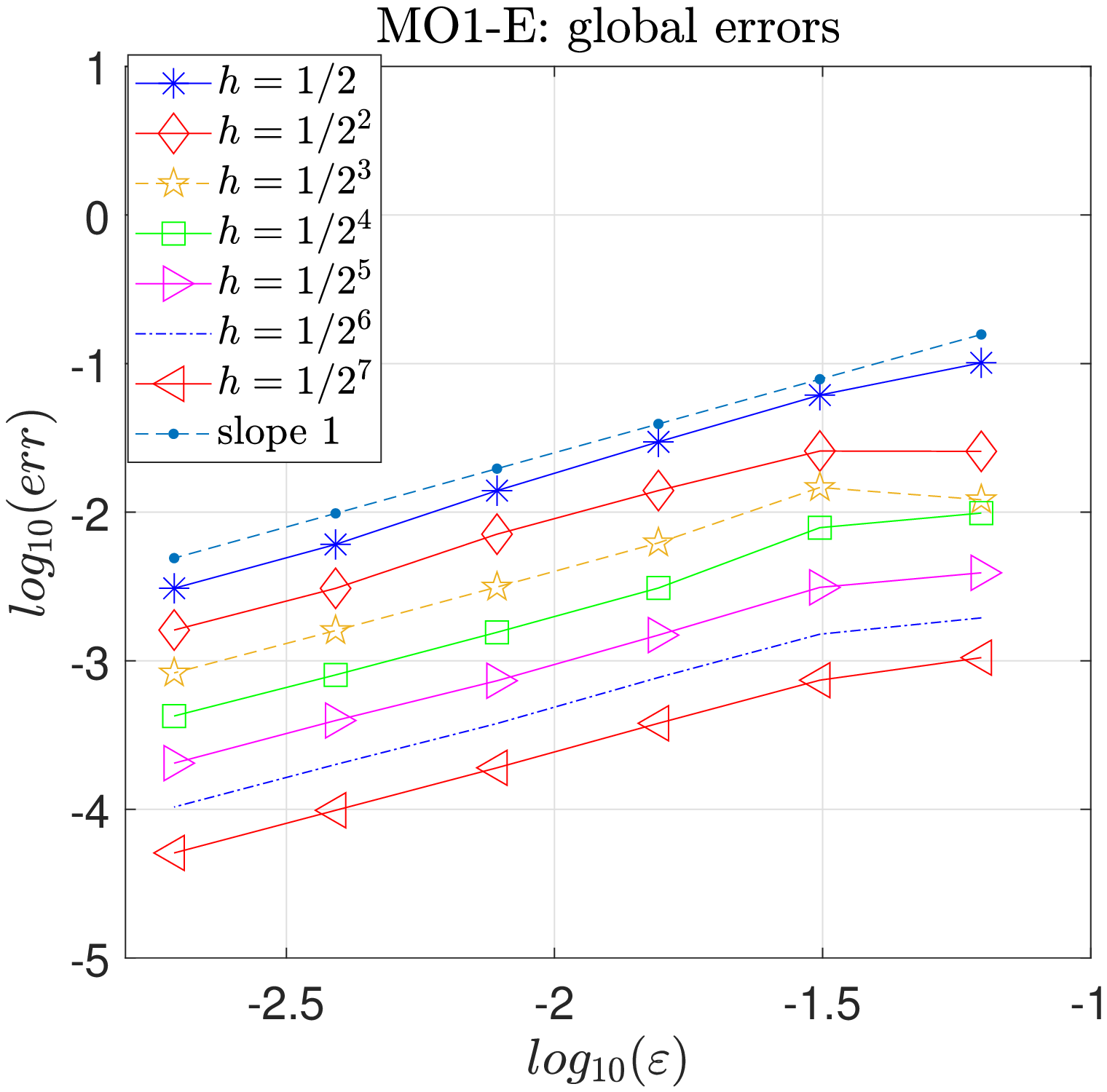,height=4.0cm,width=3.6cm}
\psfig{figure=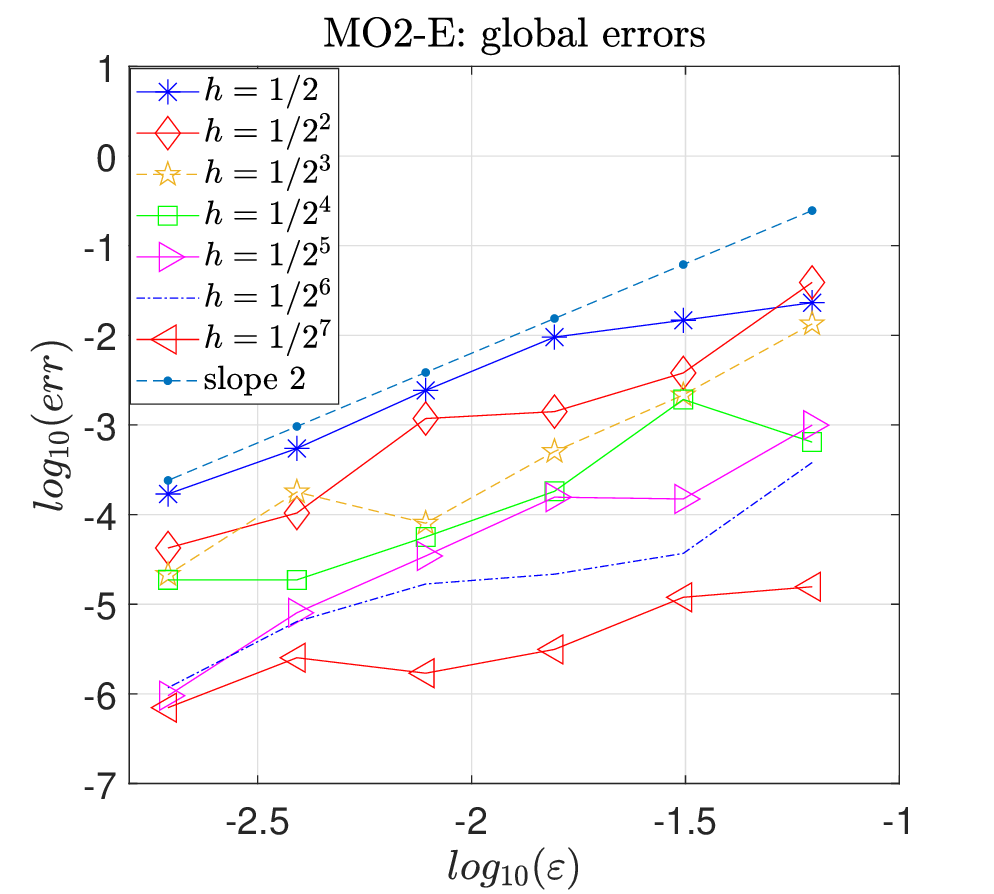,height=4.0cm,width=3.6cm}
\psfig{figure=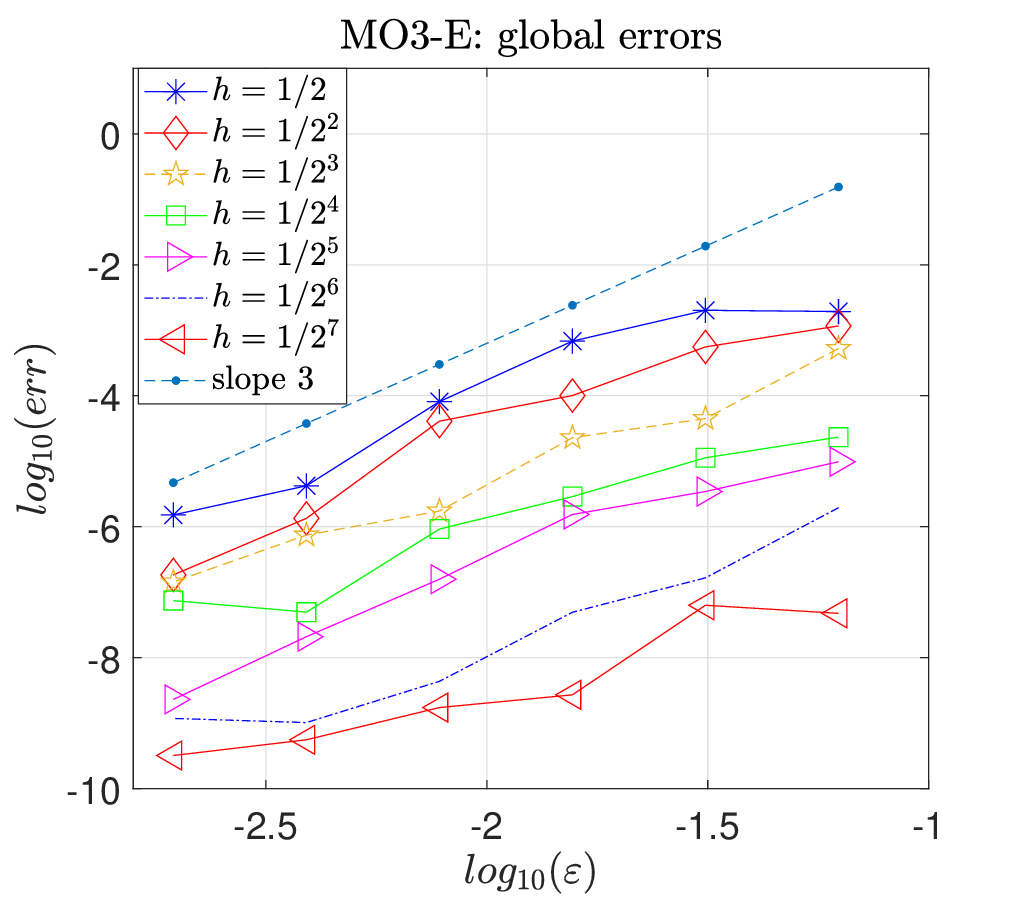,height=4.0cm,width=3.6cm}
\psfig{figure=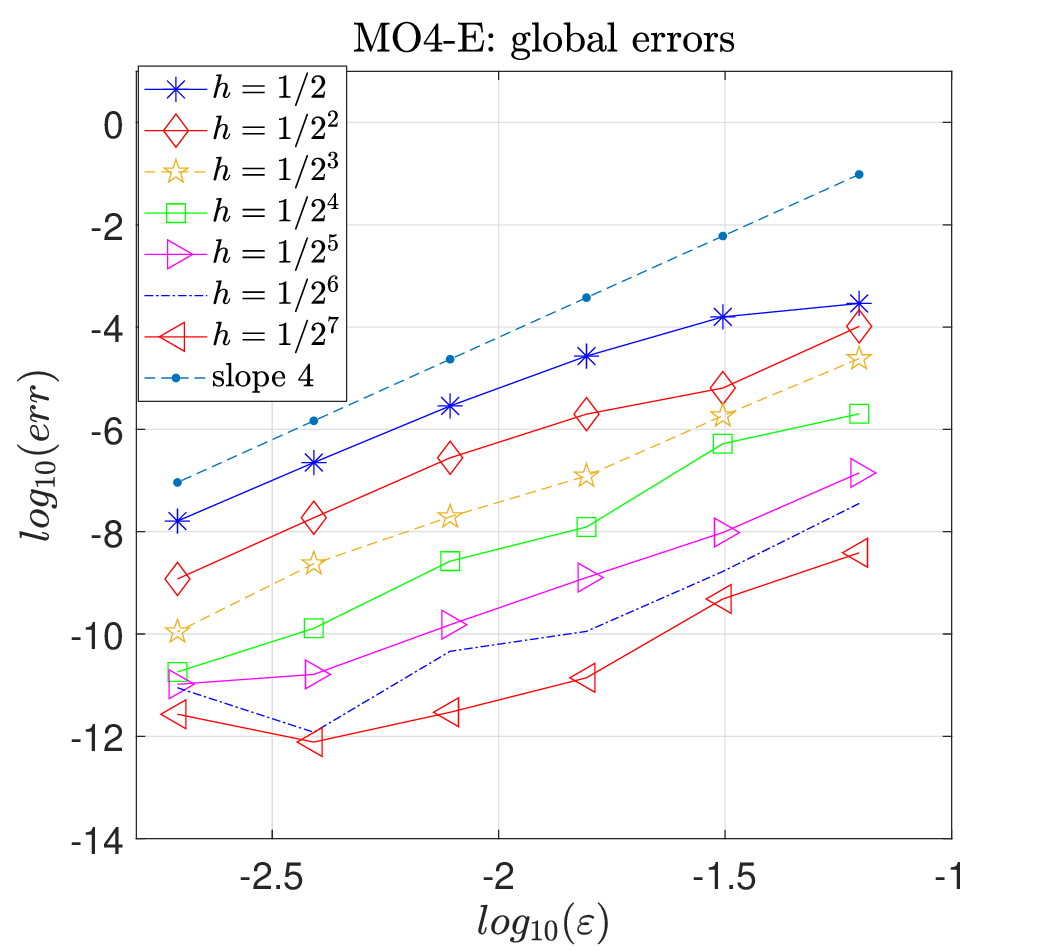,height=4.0cm,width=3.6cm}
\end{array}$$
\caption{Convergence order w.r.t. $\eps$ for the maximal ordering case \eqref{moc}:  the log-log plot of the temporal error   \eqref{err} at $1$ against  $\eps$.}\label{fig21}
\end{figure}

 {
   \textbf{Global error bounds w.r.t. $\eps$.}
   The accuracy of all the methods is shown by displaying the global errors
  \begin{equation}err= \norm{x^n-x(nh  )}/\norm{x(nh  )}+\eps \norm{v^n-v(nh  )}/\norm{v(nh )},\ \ n=1/h.\label{err}%
\end{equation}
To demonstrate the convergence order w.r.t. $\eps$, we show the temporal errors   $err$  against different values of $\eps$ in Figs. \ref{fig1} and \ref{fig21} for the general strong magnetic field \eqref{gsm} and the maximal ordering case \eqref{moc}, respectively. In the light of these results, we have the observations.

 i)  Fig.  \ref{fig1} clearly illustrates that the accuracy remains constant regardless of variations in the $\varepsilon$ value. This observation confirms that the four new integrators exhibit uniform   accuracy for general strong magnetic field \eqref{gsm}.

 ii) For the  maximal ordering magnetic field \eqref{moc}, the numerical results given in Fig. \ref{fig21} clearly show that the temporal  errors    behave like  $\mathcal{O}(\eps^r)$  for the $r$-th order integrator. The slope of the lines supports  the convergence order w.r.t. $\eps$ as stated in    \eqref{err res2} of Theorem \ref{UA thm2}.

  \textbf{Global error bounds w.r.t. $h$.}
 Then we display the temporal errors against the time stepsize $h$ in  Figs. \ref{fig2} and \ref{fig22}. It can be seen clearly that the results agree with the  theoretical time order of the four methods presented in  \eqref{err res1} of Theorem \ref{UA thm2}.

\textbf{Efficiency.}
To investigate the efficiency of our proposed integrators, we  compare the second-order scheme MO2-E with
the Boris method \cite{Boris1970} (denoted by Boris), an explicit second-order Runge-Kutta method \cite{hairer2006} (denoted by RK), the Crank-Nicolson integrator \cite{Chacon} (denoted by CN) and   an second-order asymptotic-preserving integrator  \cite{VP4} (denoted by AP). Figs. \ref{fig3}-\ref{fig32}  displays the error against the CPU time  for two different   strong magnetic fields. The results demonstrate that  MO2-E can reach the same error level with less CPU time. This advantage becomes more pronounced as the parameter $\eps$  decreases.
 Furthermore,  the efficiency of MO2-E  is enhanced when the magnetic field is maximal ordering.}

   \begin{figure}[t!]
$$\begin{array}{cc}
\psfig{figure=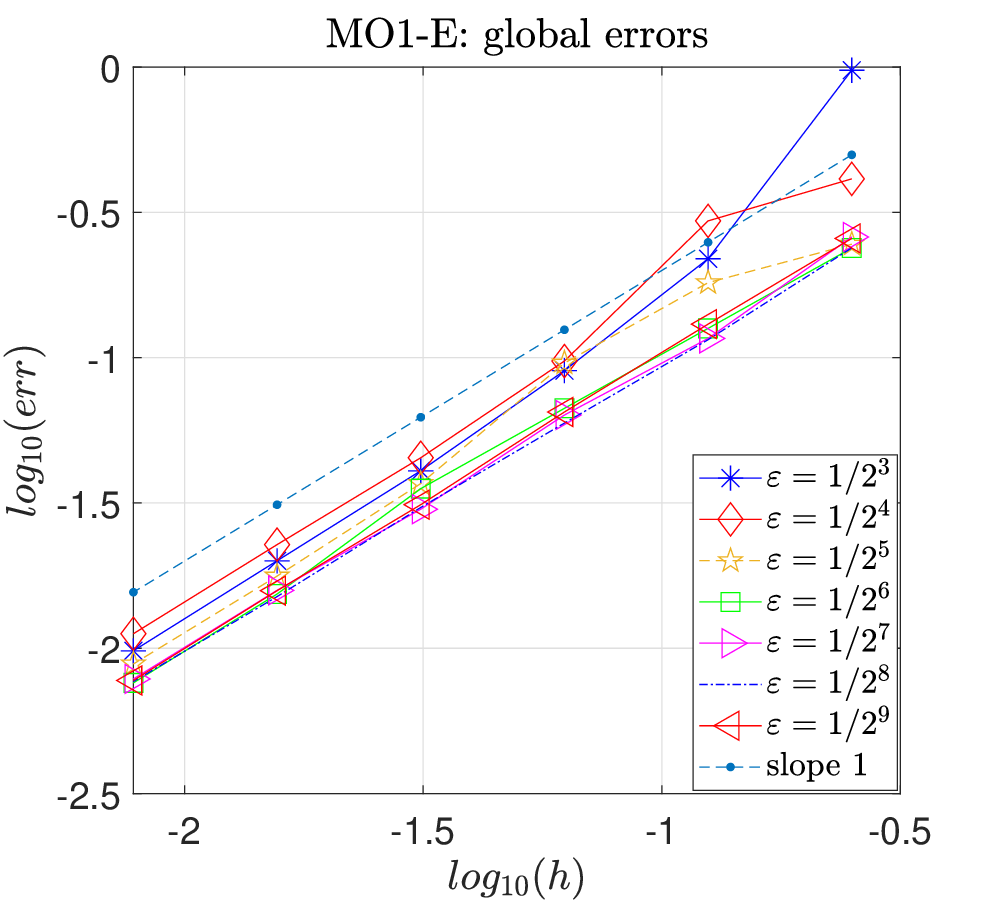,height=4.0cm,width=3.6cm}
\psfig{figure=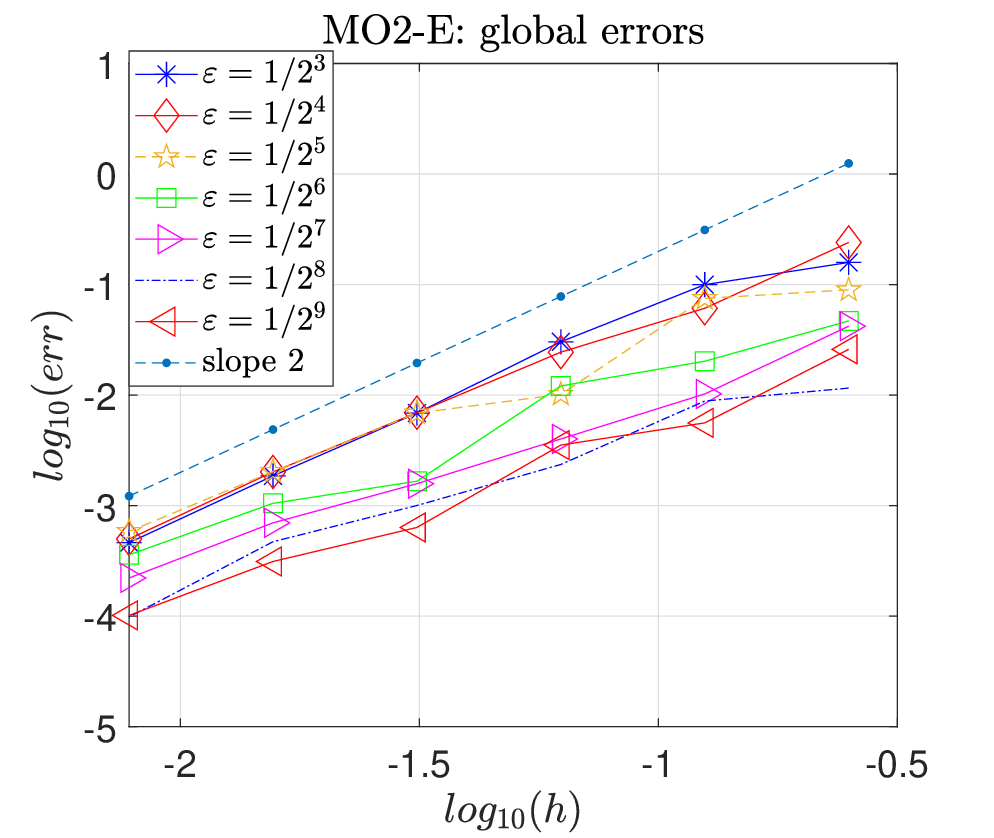,height=4.0cm,width=3.6cm}
\psfig{figure=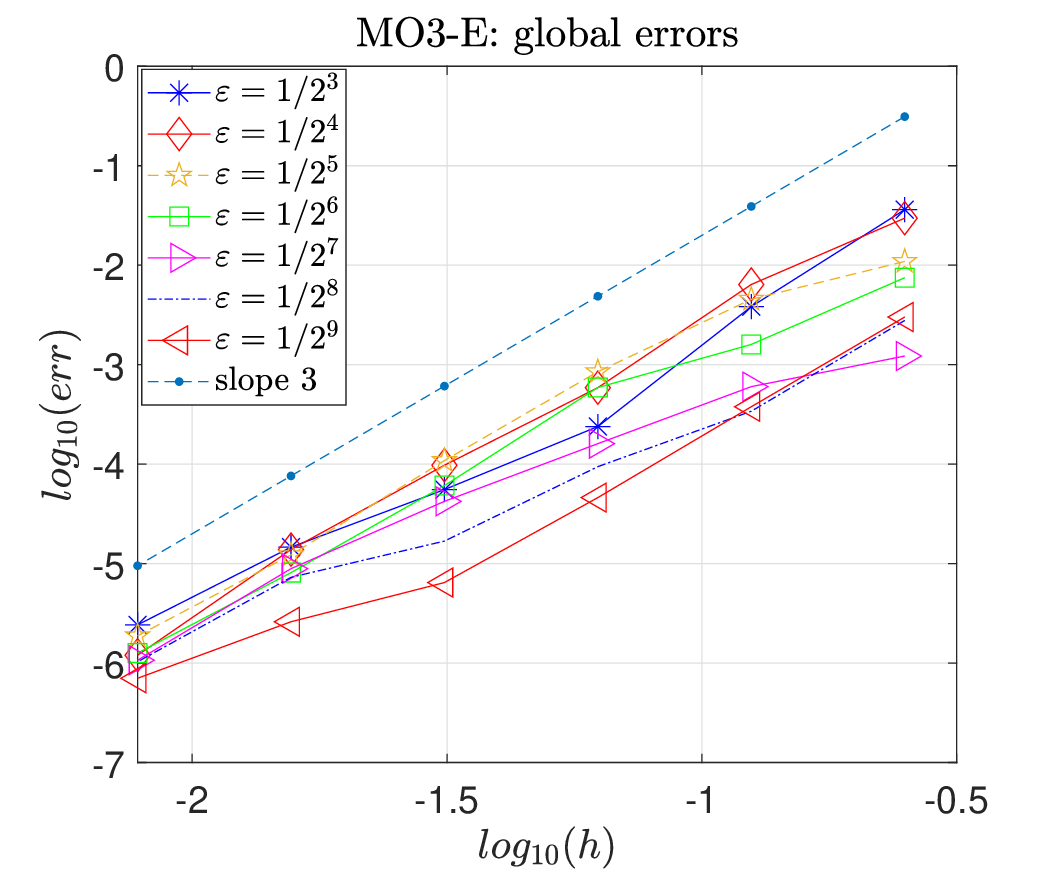,height=4.0cm,width=3.6cm}
\psfig{figure=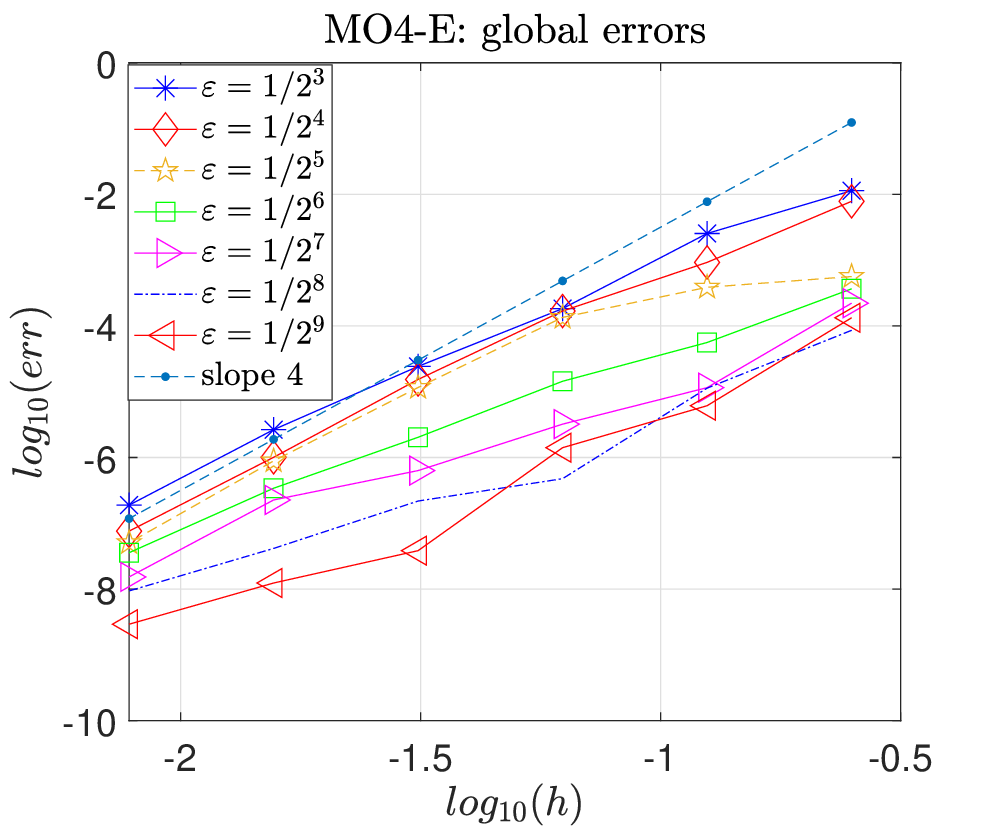,height=4.0cm,width=3.6cm}
\end{array}$$
\caption{Convergence order w.r.t. $h$ for the general strong magnetic field \eqref{gsm}: the log-log plot of the temporal error  \eqref{err} at $1$ against  $h$.}\label{fig2}
\end{figure}

   \begin{figure}[t!]
$$\begin{array}{cc}
\psfig{figure=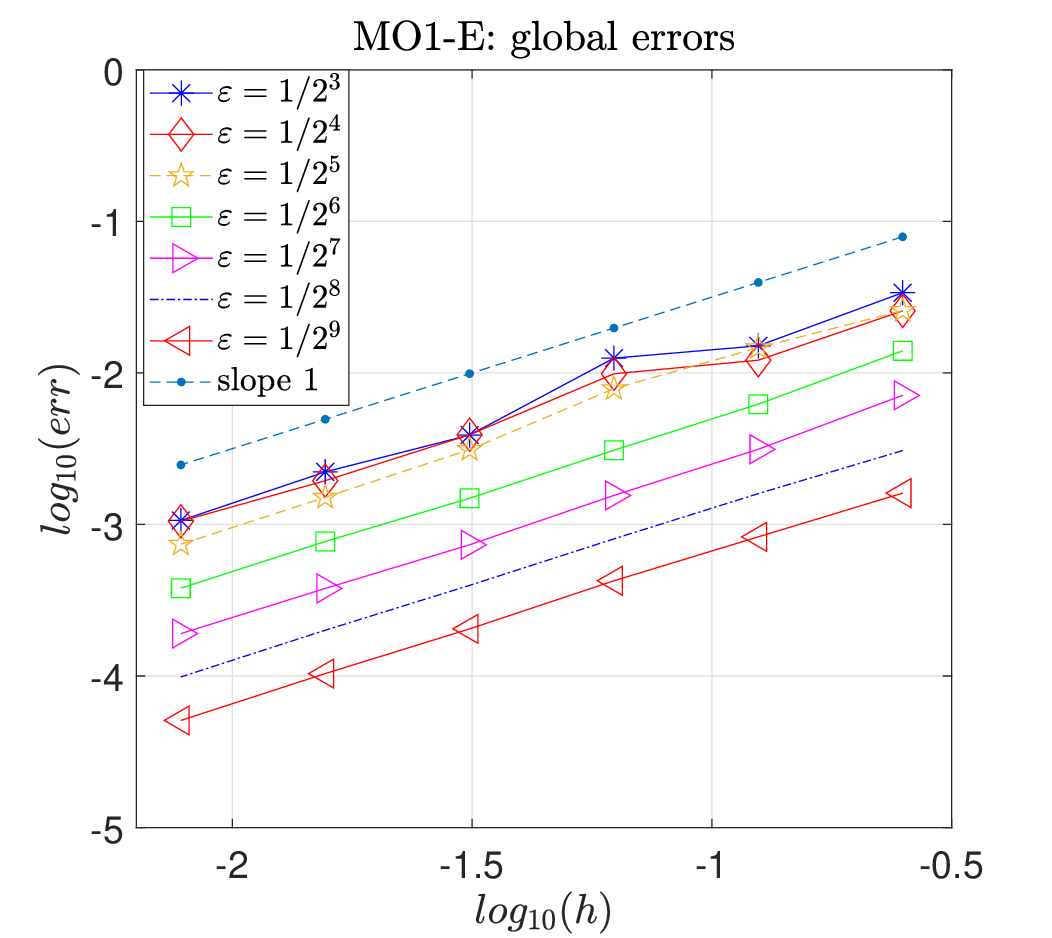,height=4.0cm,width=3.6cm}
\psfig{figure=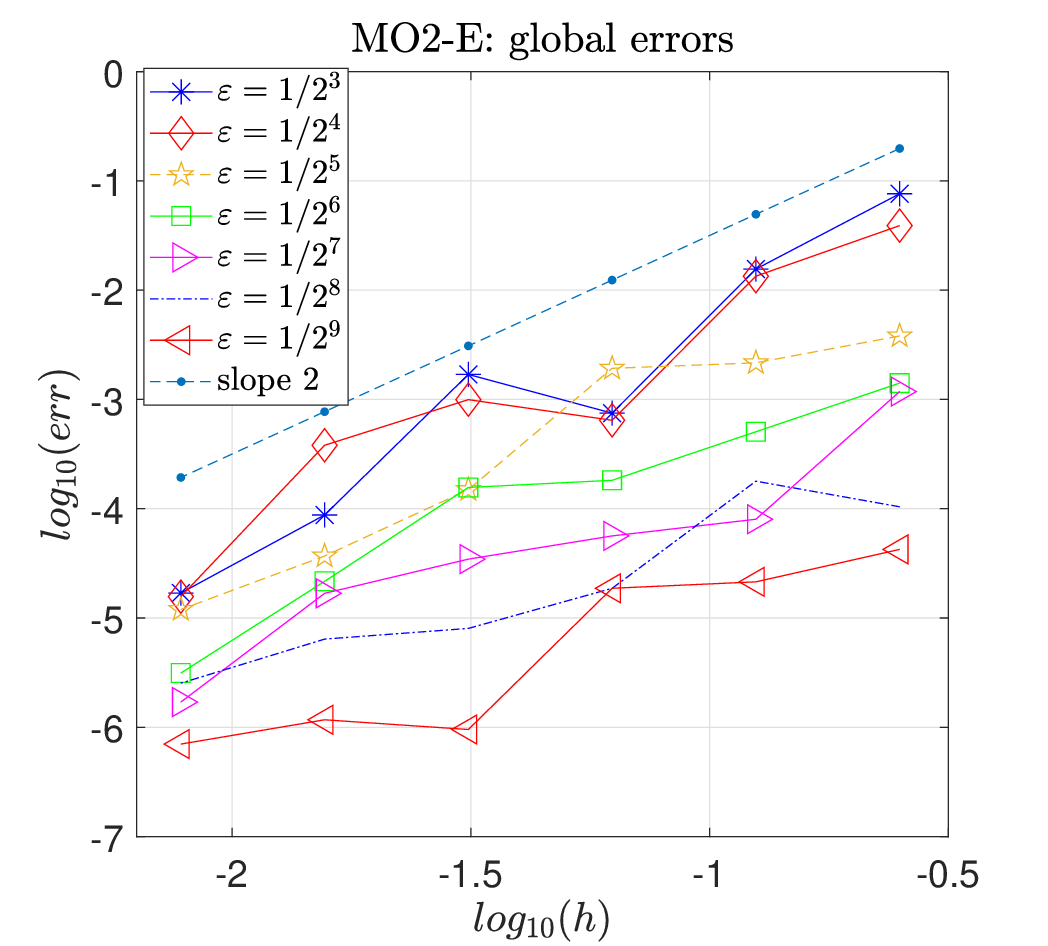,height=4.0cm,width=3.6cm}
\psfig{figure=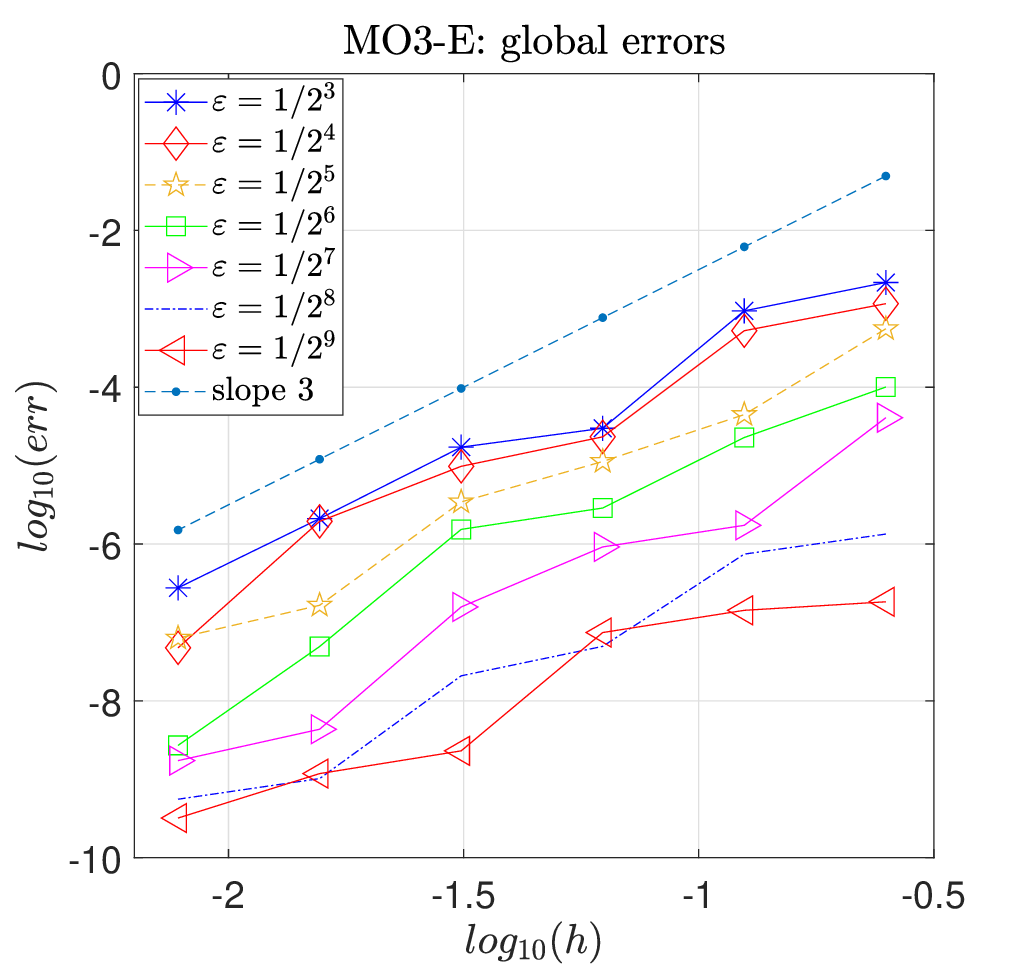,height=4.0cm,width=3.6cm}
\psfig{figure=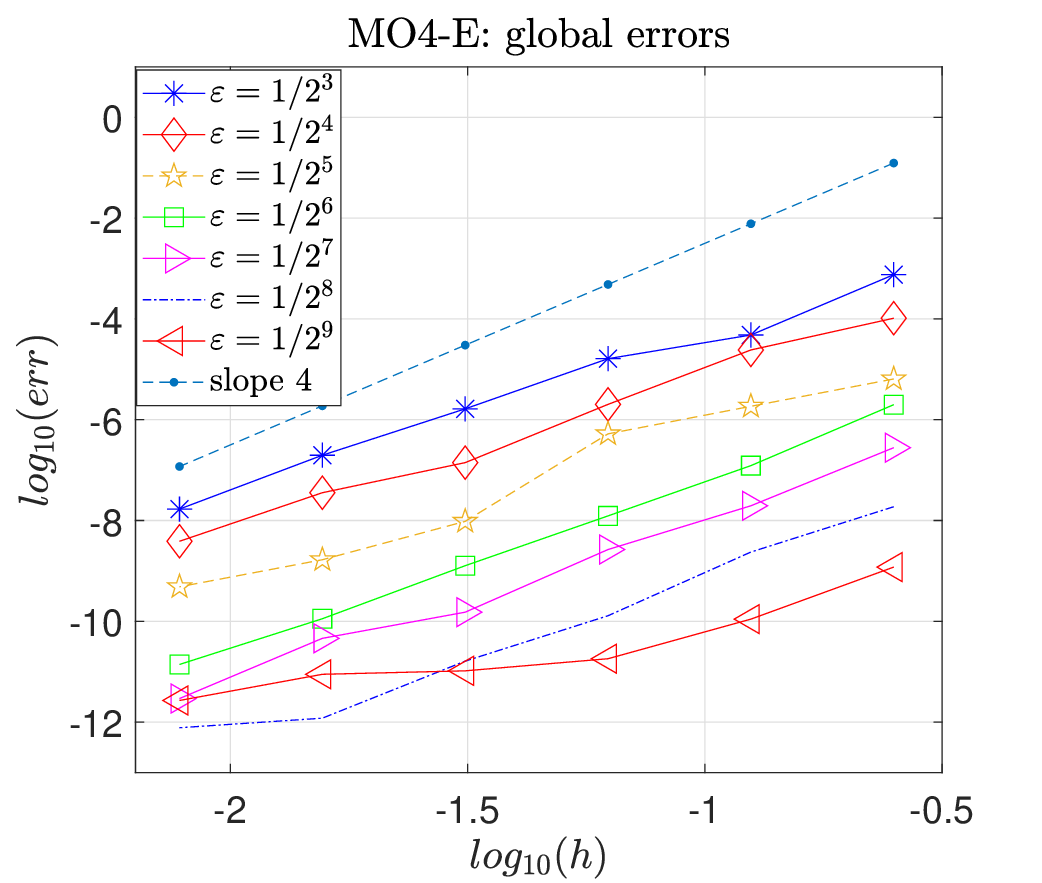,height=4.0cm,width=3.6cm}
\end{array}$$
\caption{Convergence order w.r.t. $h$ for the maximal ordering case \eqref{moc}: the log-log plot of the temporal error   \eqref{err} at $1$ against  $h$.}\label{fig22}
\end{figure}

\section{Conclusion}\label{sec:con}
In this paper, we proposed and researched  numerical integrators of  three-dimensional charged-particle dynamics (CPD) in a strong nonuniform magnetic field. We introduced a new methodology in the formulation of the intergrators and rigorously studied their error estimates. It was shown that four proposed algorithms can have uniform accuracy of orders up to four. { Improved uniform error bounds were obtained  when  these integrators are used to numerically integrate
the  CPD under a maximal ordering strong magnetic field.}  Some numerical results were displayed to demonstrate the { error and efficiency}  behaviour of the proposed methods.

   \begin{figure}[t!]
$$\begin{array}{cc}
\psfig{figure=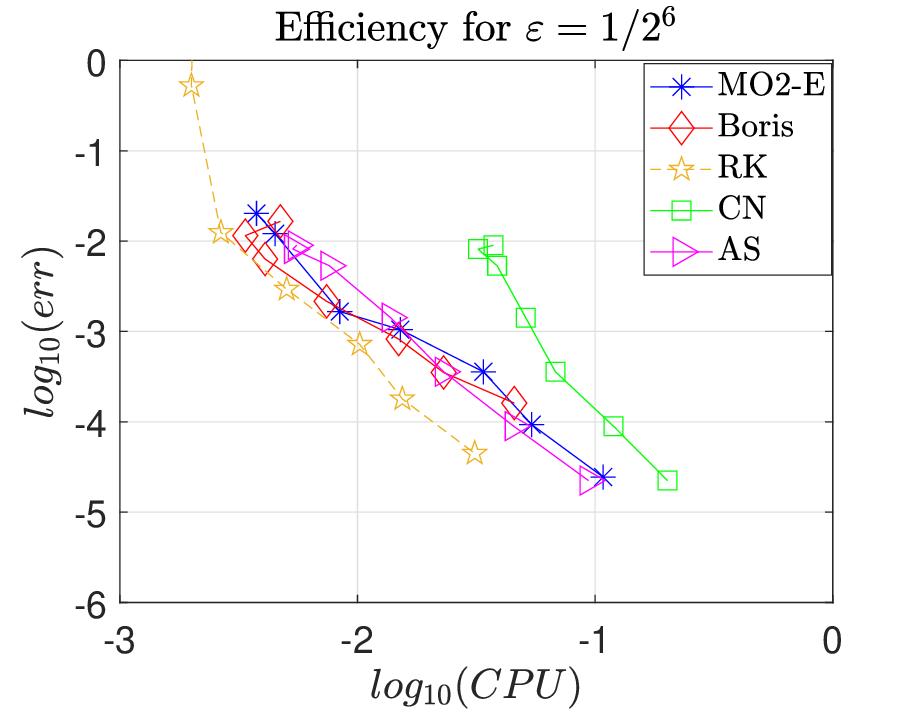,height=4.0cm,width=3.6cm}
\psfig{figure=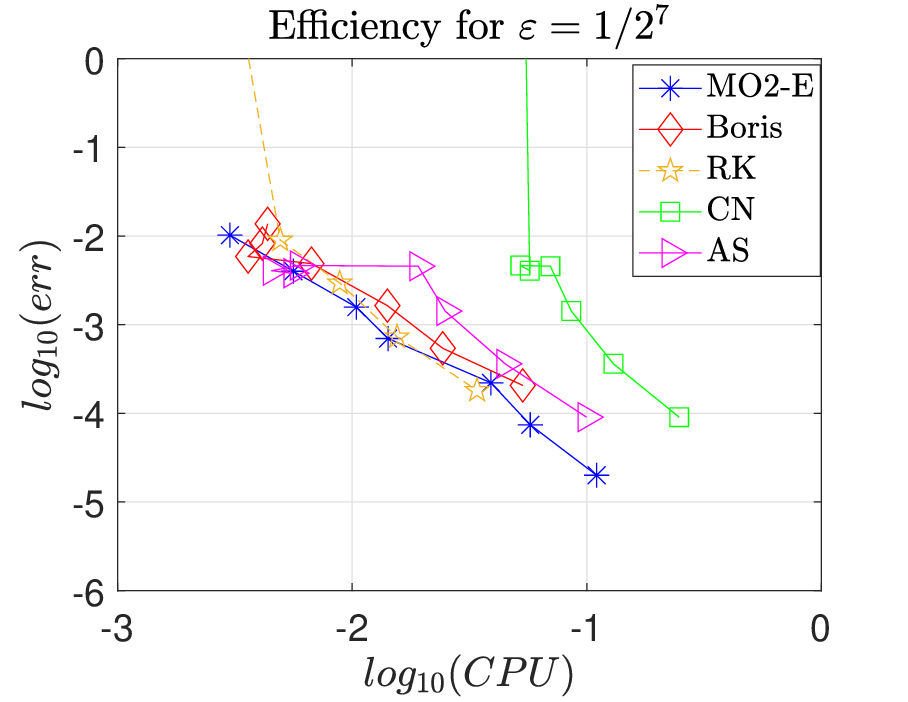,height=4.0cm,width=3.6cm}
\psfig{figure=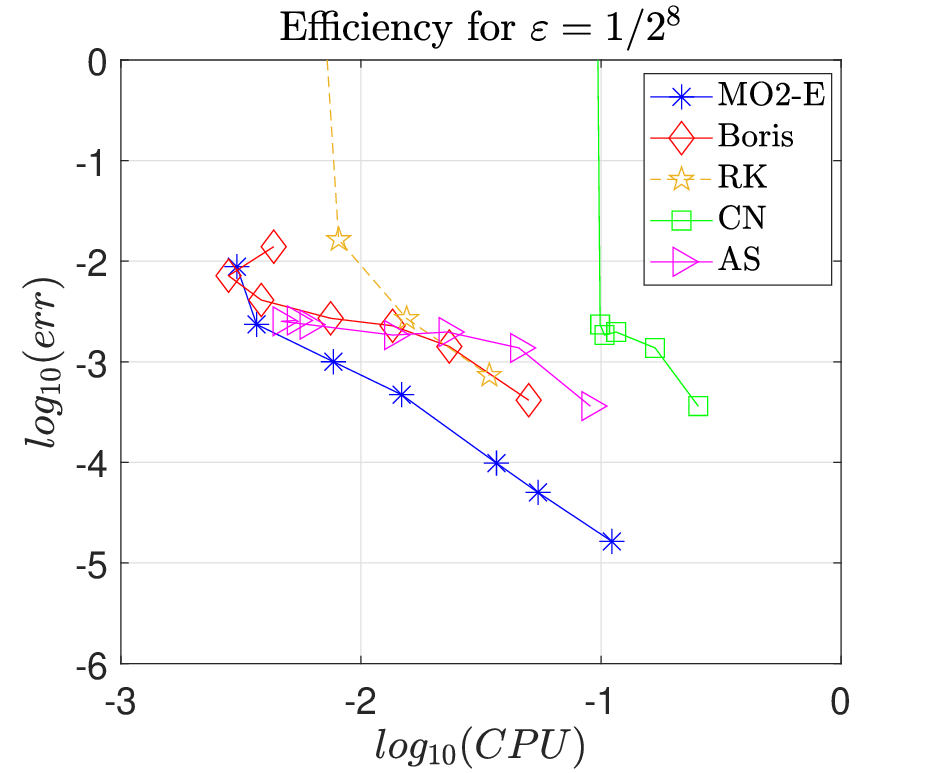,height=4.0cm,width=3.6cm}
\psfig{figure=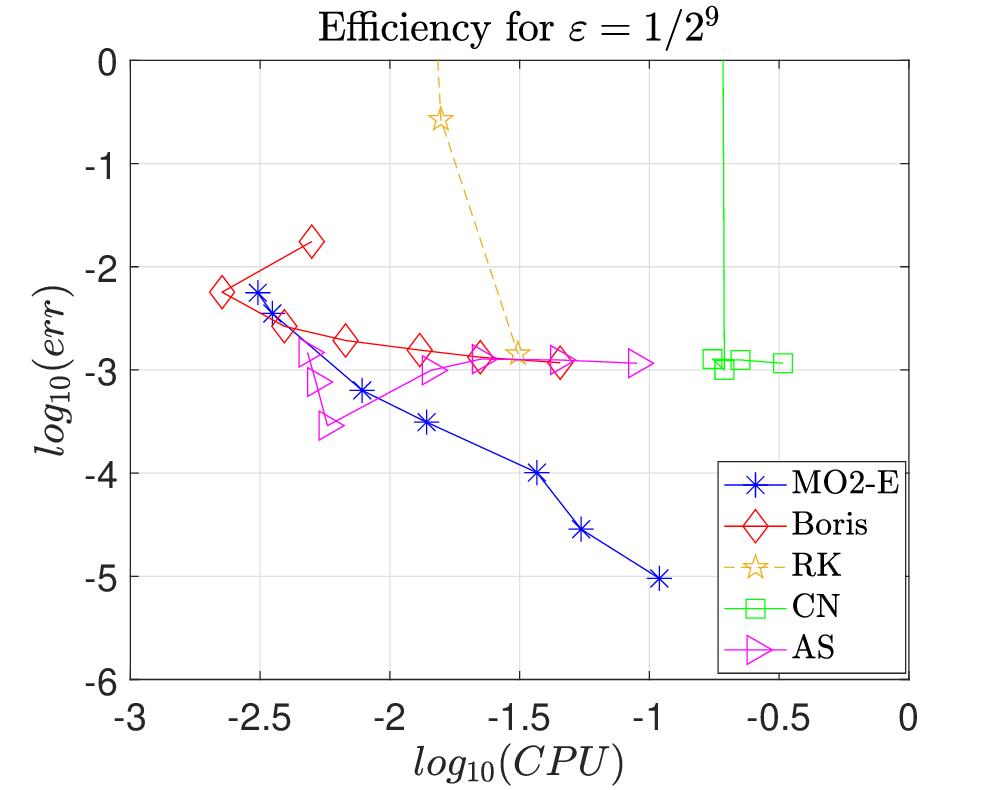,height=4.0cm,width=3.6cm}
\end{array}$$
\caption{Efficiency for the general strong magnetic field \eqref{gsm}:  the log-log plot of the temporal error  $err$    at $1$ against  CPU time.}\label{fig3}
\end{figure}

   \begin{figure}[t!]
$$\begin{array}{cc}
\psfig{figure=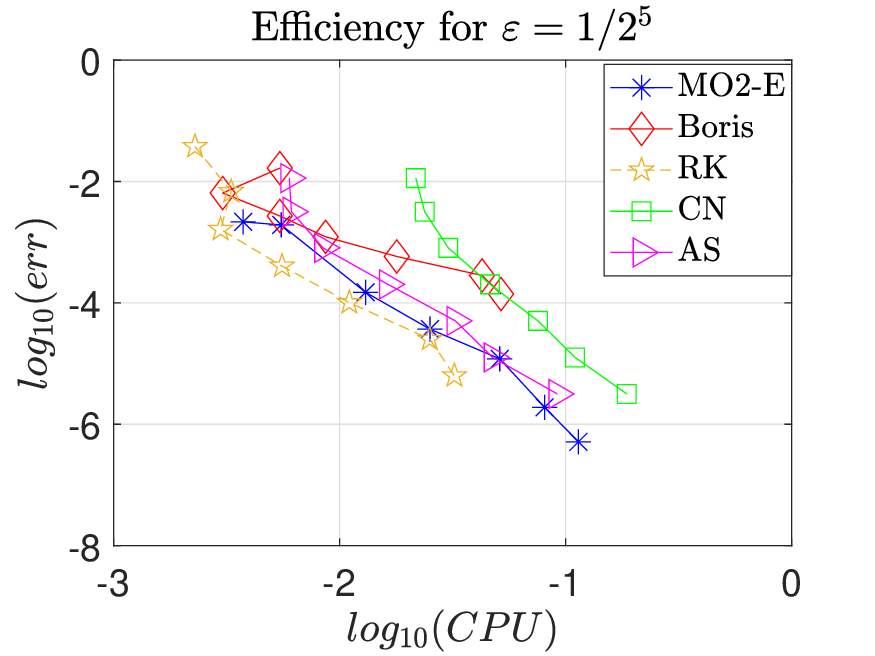,height=4.0cm,width=3.6cm}
\psfig{figure=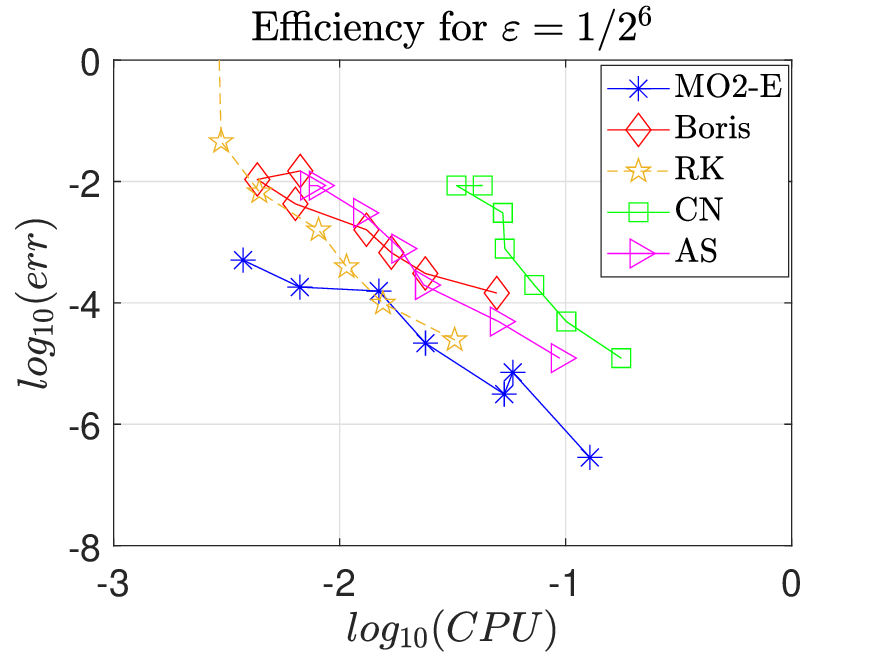,height=4.0cm,width=3.6cm}
\psfig{figure=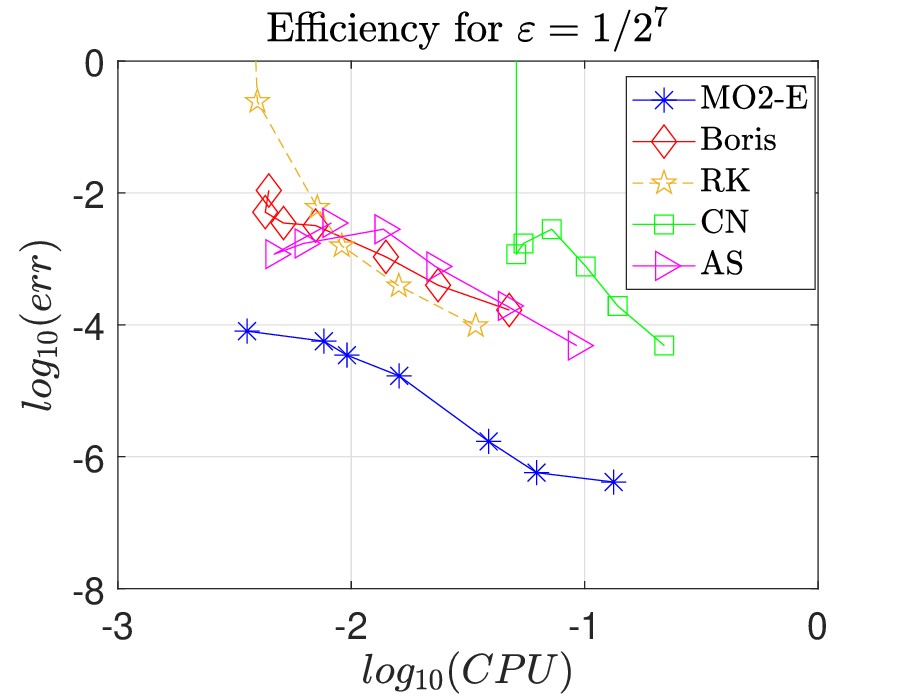,height=4.0cm,width=3.6cm}
\psfig{figure=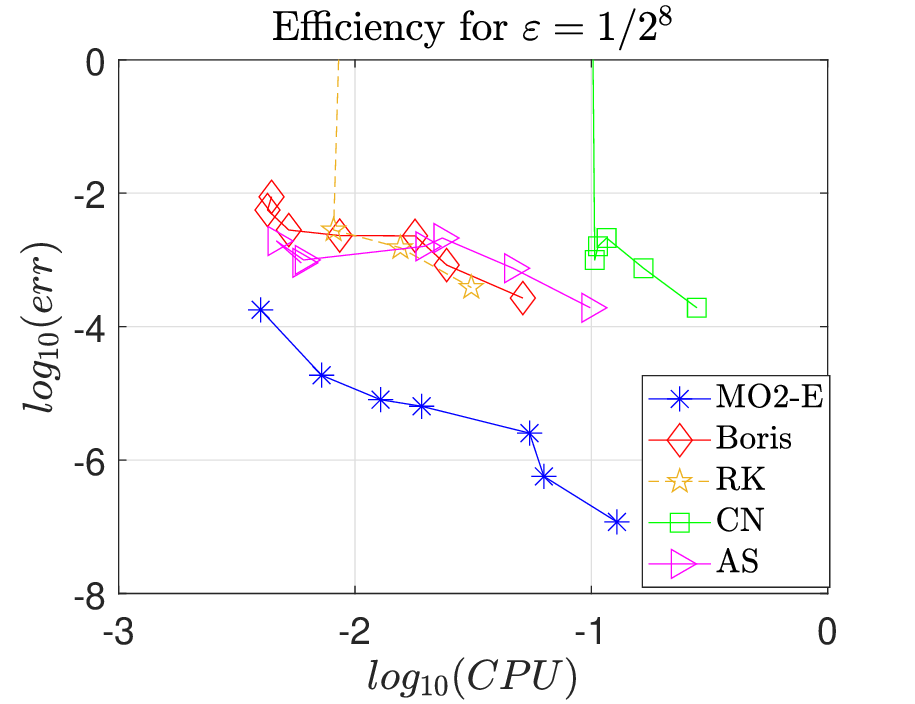,height=4.0cm,width=3.6cm}
\end{array}$$
\caption{Efficiency for the maximal ordering case \eqref{moc}::  the log-log plot of the temporal error  $err$   at $1$ against  CPU time.}\label{fig32}
\end{figure}

\section*{Acknowledgements}
The authors are grateful to   the two anonymous reviewers for their very valuable suggestions, which help improve
this paper significantly.  


\begin{thebibliography}{99}
\bibitem {Arnold97} {\sc V.I. Arnold, V.V. Kozlov, and  A.I. Neishtadt}, \emph{Mathematical
Aspects of Classical and Celestial Mechanics}, Springer, Berlin, 1997.


 \bibitem{Bao23}
{\sc W. Bao, Y. Cai,  and   Y. Feng},  \emph{Improved uniform error bounds of the time-splitting methods for the long-time (nonlinear) Schr\"{o}dinger equation}, Math. Comp.  92 (2023),  pp. 1109-1139.


   \bibitem{Bao21}
{\sc W. Bao, Y. Cai,  and   Y. Feng}, \emph{Improved uniform error bounds on time-splitting methods for long-time dynamics of the nonlinear Klein-Gordon equation with weak nonlinearity}, SIAM J. Numer. Anal.  60 (2022),  pp. 1962-1984.



 \bibitem {Benettin94}{\sc G. Benettin and  P. Sempio}, \emph{Adiabatic invariants and trapping
of a point charge in a strong nonuniform magnetic field},
Nonlinearity 7 (1994), pp. 281-304.

\bibitem {Birdsall} {\sc C.K. Birdsall and  A.B. Langdon}, \emph{Plasma physics via computer simulation}, Series in plasma physics, Taylor
$\&$ Francis, New York, 2005.


\bibitem {Boris1970} {\sc J.P. Boris}, \emph{Relativistic plasma simulation-optimization of a hybrid
code}, Proceeding of Fourth Conference on Numerical Simulations of
Plasmas (1970), pp. 3-67.
%

 \bibitem{L. Brugnano2019} {\sc L. Brugnano, J.I. Montijano,  and L. R\'{a}ndz}, \emph{High-order energy-conserving line integral methods for charged particle dynamics}, J. Comput. Phys. 396 (2019), pp. 209-227.

\bibitem{scaling1}
{\sc A.J. Brizard and  T.S. Hahm}, \emph{Foundations of nonlinear gyrokinetic theory}, Rev. Mod. Phys. 79
(2007), pp. 421-468.




\bibitem{VP1}{\sc Ph. Chartier, N. Crouseilles, M. Lemou, F. M\'ehats, and  X. Zhao}, \emph{Uniformly accurate methods for Vlasov equations with non-homogeneous strong magnetic field}, Math. Comp. 88 (2019), pp. 2697-2736.

    \bibitem{Zhao}
{\sc Ph. Chartier, N. Crouseilles, M. Lemou, F. M\'ehats, and  X. Zhao},
\emph{Uniformly accurate methods for three dimensional Vlasov equations under strong magnetic field with varying direction}, SIAM J. Sci. Comput.  42 (2020), pp. B520-B547.


\bibitem{VP2} {\sc Ph. Chartier, N. Crouseilles, and  X. Zhao}, \emph{Numerical methods for the two-dimensional Vlasov-Poisson equation in the finite Larmor radius approximation regime}, J. Comput. Phys. 375 (2018),  pp. 619-640.


    \bibitem{autoUA}  {\sc Ph. Chartier, M. Lemou, F. M\'{e}hats, and    X. Zhao},
 \emph{Derivative-free high-order uniformly accurate schemes for highly-oscillatory systems}, IMA J. Numer. Anal.,  42 (2022), pp.   1623-1644.




\bibitem{CA16} {\sc G. Colonna and  A. D. Angola},  \emph{Plasma Modeling}, IOP Publishing, Bristol, UK, 2016.

\bibitem{ADD5} {\sc N. Crouseilles,  L. Einkemmer, and  M. Prugger},  \emph{An exponential integrator for the drift-kinetic model}, Comput. Phys. Comm. 224 (2018), pp.  144-153.


\bibitem{CPC} {\sc N. Crouseilles, S.A. Hirstoaga, and   X. Zhao}, \emph{Multiscale Particle-In-Cell methods
and comparisons for the long-time two-dimensional Vlasov-Poisson equation with
strong magnetic field}, Comput. Phys. Comm. 222 (2018), pp. 136–151.



\bibitem{VP3} {\sc N. Crouseilles, M. Lemou, F. M\'ehats, and  X. Zhao}, \emph{Uniformly accurate Particle-in-Cell method for the long time two-dimensional Vlasov-Poisson equation with uniform strong magnetic field}, J. Comput. Phys. 346 (2017), pp. 172-190.




\bibitem{ADD4} {\sc L. Einkemmer,  M. Tokman, and  J. Loffeld},  \emph{On the performance of exponential integrators for problems in magnetohydrodynamics}, J. Comput. Phys. 330 (2016), pp.  550-565.

\bibitem{VP4} {\sc F. Filbet and  L.M. Rodrigues}, \emph{Asymptotically stable particle-in-cell methods for the Vlasov-Poisson system with a strong external magnetic field}, SIAM J. Numer. Anal. 54 (2016), pp. 1120-1146.

\bibitem{VP5}   {\sc F. Filbet and  L.M. Rodrigues}, \emph{Asymptotically preserving particle-in-cell methods for inhomogeneous strongly magnetized plasmas}, SIAM J. Numer. Anal. 55 (2017), pp. 2416-2443.

  \bibitem {A1} {\sc F. Filbet and L.M. Rodrigues} \emph{Asymptotics of the three dimensional Vlasov equation in the large magnetic field limit}, Journal Ecole Polytechnique,  7  (2020), pp.  1009-1067.

\bibitem {A2} {\sc F. Filbet and L.M. Rodrigues}, \emph{Asymptotically preserving particle methods for strongly magnetized plasmas in a torus}, J. Comput. Phys.  480 (2023), 112015.





\bibitem{VP-filbet} {\sc F. Filbet, T. Xiong,  and E. Sonnendr\"{u}cker}, \emph{On the Vlasov-Maxwell system with a strong magnetic field}, SIAM J. Appl. Math. 78 (2018), pp. 1030-1055.


 \bibitem{VP8}
{\sc E. Fr\'{e}nod, S. Hirstoaga, M. Lutz, and  E. Sonnendr\"{u}cker}, \emph{Long time behavior of an exponential integrator for a Vlasov-Poisson system with strong magnetic field}, Commun. Comput. Phys. 18 (2015), pp. 263-296.




 \bibitem {ADD2}   {\sc L. Gauckler, J. Lu, J. Marzuola, F. Rousset, and  K. Schratz}, \emph{Trigonometric integrators for quasilinear wave equations}, Math. Comp. 88 (2019),  pp. 717-749.

\bibitem {Hairer2017-1}{\sc E. Hairer and  Ch. Lubich}, \emph{Energy behaviour of the Boris method for
charged-particle dynamics},  BIT 58  (2018),  pp. 969-979.

\bibitem {Hairer2017-2}{\sc E. Hairer and  Ch. Lubich}, \emph{Symmetric multistep methods for
charged-particle dynamics},  SMAI J. Comput. Math. 3 (2017), pp. 205-218.

\bibitem {Hairer2018}{\sc E. Hairer and  Ch. Lubich}, \emph{Long-term analysis of a variational integrator for
charged-particle dynamics in a strong magnetic field},
 Numer. Math. 144 (2020), pp. 699-728.

\bibitem {Hairer2022}{\sc E. Hairer, Ch. Lubich,  and Y. Shi},  \emph{Large-stepsize integrators for charged-particle dynamics over multiple time scales}, Numer. Math.  151 (2022), pp. 659-691.



\bibitem {Hairer2023}{\sc E. Hairer, Ch. Lubich,  and Y. Shi},  \emph{Leapfrog methods for relativistic charged-particle dynamics}, SIAM J. Numer. Anal. 61 (2023), pp. 2844-2858.

  \bibitem{lubich19}{\sc E. Hairer, Ch. Lubich,  and  B. Wang}, \emph{A filtered Boris algorithm for
charged-particle dynamics in a strong magnetic field},
Numer. Math. 144 (2020), pp. 787-809.


\bibitem {hairer2006} {\sc E. Hairer, Ch. Lubich, and G. Wanner}, \emph{Geometric Numerical
Integration: Structure-Preserving Algorithms for Ordinary
Differential Equations}, 2nd edn.  Springer-Verlag, Berlin,
Heidelberg, 2006.


\bibitem {He2017}{\sc Y. He, Z. Zhou,
Y. Sun, J. Liu,  and H. Qin}, \emph{Explicit K-symplectic algorithms for charged
particle dynamics}, Phys. Lett. A 381 (2017), pp. 568-573.


\bibitem{Hochbruck1999}
{\sc M.  Hochbruck and   Ch. Lubich},    \emph{A Gautschi-type method for oscillatory second-order differential equations}, Numer. Math.  83
(1999), pp.  403-426.

 \bibitem{Ostermann06}
{\sc M. Hochbruck and   A. Ostermann},  \emph{Explicit exponential Runge--Kutta methods for semilinear parabolic problems}, SIAM J. Numer. Anal. 43 (2006), pp. 1069-1090.

   \bibitem{Ostermann}
{\sc M. Hochbruck and  A. Ostermann}, \emph{Exponential integrators}, Acta Numer. 19 (2010), pp. 209-286.


\bibitem {Ostermann15}{\sc C. Knapp, A. Kendl,  A. Koskela, and  A. Ostermann}, \emph{Splitting methods for time integration of trajectories in combined electric and magnetic fields}, Phys. Rev. E  92 (2015),  063310.


\bibitem {VP7} {\sc M. Kraus, K. Kormann, P. Morrison, and  E. Sonnendr\"ucker}, \emph{GEMPIC: geometric electromagnetic Particle In Cell methods}, J. Plas. Phys. 4 (2017),  83.

    \bibitem{ADD3}
{\sc B. Li, K. Schratz,  and   F. Zivcovich}, \emph{A second-order low-regularity correction of Lie splitting for the semilinear Klein-Gordon equation},
ESAIM: Math. Model. Numer. Anal. 57 (2023),  pp. 899-919.


    \bibitem{ADD6} {\sc  Y. Miyatake}, \emph{An energy-preserving exponentially-fitted continuous stage Runge-Kutta method for Hamiltonian systems},  BIT Numer. Math. 54 (2014),  pp. 1-23.


      \bibitem {ADD1}   {\sc A. Ostermann and  K. Schratz}, \emph{Low regularity exponential-type integrators for semilinear Schr\"{o}dinger     equations}, Found. Comput. Math. 18 (2018),  pp. 731-755.


%






\bibitem{scaling2}
{\sc S. Possanner}, \emph{Gyrokinetics from variational averaging: existence and error bounds}, J. Math. Phys.
59 (2018), 082702.





\bibitem{Chacon} {\sc L.F. Ricketson and  L. Chac\'{o}n}, \emph{An energy-conserving and asymptotic-preserving charged-particle orbit implicit time integrator for arbitrary electromagnetic fields}, J. Comput. Phys. 418 (2020), 109639.



    \bibitem{Shen} {\sc J. Shen, T. Tang, and  L. Wang}, \emph{Spectral Methods: Algorithms, Analysis, Applications}, Springer, Berlin, 2011.

\bibitem{SonnendruckerBook}
{\sc E. Sonnendr\"{u}cker}, \emph{Numerical Methods for Vlasov Equations}, Lecture notes, 2016.

 \bibitem {Tao2016}{\sc M. Tao}, \emph{Explicit high-order symplectic
integrators for charged particles in general electromagnetic fields},
J. Comput. Phys.  327 (2016), pp. 245-251.


\bibitem{WJ23}
 {\sc B. Wang and  Y.-L. Jiang}, \emph{Semi-discretization and full-discretization with optimal accuracy for charged-particle dynamics in a strong nonuniform magnetic field}, ESAIM: Math. Model. Numer. Anal.   57 (2023), pp. 2427-2450.


    \bibitem{WZ}
{\sc B. Wang and   X. Zhao}, \emph{Error estimates of some splitting schemes for
charged-particle dynamics under strong magnetic field}, SIAM J. Numer. Anal. 59 (2021) pp. 2075-2105.


\bibitem{WZ23}
{\sc B. Wang and   X. Zhao}, \emph{Geometric two-scale integrators for highly oscillatory system: uniform accuracy and near conservations}, SIAM J. Numer. Anal. 61 (2023), pp.  1246-1277.




 \bibitem {Webb2014}{\sc S.D. Webb}, \emph{Symplectic integration of
magnetic systems}, J. Comput. Phys. 270 (2014), pp. 570-576.

\bibitem {Zhang2016} {\sc R. Zhang, H.
Qin, Y. Tang, J. Liu, Y. He,  and  J. Xiao}, \emph{Explicit symplectic
algorithms based on generating functions for charged particle
dynamics}, Phys. Rev. E 94 (2016),  013205.

%
%
%
%
%
%
%

\end{thebibliography}
\end{document}